\documentclass[12pt]{article}
\usepackage{amssymb}
\usepackage{enumerate}
\usepackage{amsfonts}
\usepackage{amsmath}
\usepackage{amsthm}
\usepackage{psfrag} 
\usepackage{epsfig}
\usepackage{subfigure}
\usepackage{graphicx}
\usepackage{color}
\usepackage{amssymb}
\usepackage{epstopdf}
\usepackage{amsbsy}
\usepackage{latexsym}
\usepackage{moreverb}                      
\usepackage{textcomp}
\usepackage{multicol}
\usepackage{algorithm}
\usepackage{algorithmic}
\usepackage{enumitem}
\usepackage{dsfont}

\newtheorem{lemma}{Lemma}
\newtheorem{remark}{Remark}
\newtheorem{definition}{Definition}

\newtheorem{proposition}{Proposition}

\newcommand{\argmin}{\operatornamewithlimits{argmin}}

\usepackage[top=2.6cm,bottom=2.6cm,left=2cm,right=2cm]{geometry}

\usepackage[T1]{fontenc}
\usepackage[utf8]{inputenc}

\usepackage[affil-it]{authblk}
\usepackage{lmodern}


\usepackage{sectsty}
\sectionfont{\fontsize{13}{15}\selectfont}
\subsectionfont{\fontsize{12}{15}\selectfont}
\begin{document}

\title{Hierarchical Low-Rank Approximation of Regularized Wasserstein Distance}

\author[1]{Mohammad Motamed\thanks{motamed@math.unm.edu}}

\affil[1]{Department of Mathematics and Statistics, The University of New Mexico, Albuquerque, USA}
\maketitle

\begin{abstract}

Sinkhorn divergence is a measure of dissimilarity between two
probability measures. It is obtained through adding
an entropic regularization term to Kantorovich's optimal transport
problem and can hence be viewed as an entropically regularized Wasserstein distance. 
Given two discrete probability vectors in the $n$-simplex and 
supported on two bounded subsets of ${\mathbb R}^d$, 
we present a fast method for computing Sinkhorn
divergence when the cost matrix can be decomposed into a $d$-term sum
of asymptotically smooth Kronecker product factors. 
%
%
%
The method combines Sinkhorn's matrix
scaling iteration with a low-rank hierarchical representation of the
scaling matrices to achieve a near-linear complexity ${\mathcal O}(n
\log^3 n)$.  
This provides a fast and easy-to-implement algorithm for computing
Sinkhorn divergence, enabling its applicability to large-scale optimization problems, where the computation of classical
Wasserstein metric is not feasible. 
We present a numerical example related to signal processing to demonstrate the applicability of
quadratic Sinkhorn divergence in comparison with quadratic
Wasserstein distance and to verify the accuracy and
efficiency of the proposed method. 
\end{abstract}

{\bf keywords} 
optimal transport, Wasserstein metric, Sinkhorn divergence,
hierarchical matrices


\section{Introduction}
\label{sec:intro}

Wasserstein distance is a measure of dissimilarity between two probability
measures, arising from optimal transport; see e.g. \cite{Villani:03,Villani:09}. Due to its desirable convexity and convergence features, Wasserstein
distance is considered as an important measure of dissimilarity
in many applications, ranging from computer vision and machine learning to Seismic and
Bayesian inversion; 
see e.g. \cite{Peyre_Cuturi:19,engquist2016optimal,Wasserstein_Bjorn,Motamed_Appelo:19}.
The application of Wasserstein metric may however be
limited to cases where the probability measures are supported on
low-dimensional spaces, because its numerical computation can quickly become prohibitive
as the dimension increases; see e.g. \cite{Peyre_Cuturi:19}.

Sinkhorn divergence \cite{Cuturi:13} is another measure of dissimilarity related to Wasserstein distance. It is obtained through adding an entropic regularization
term to the Kantorovich formulation of optimal transport problem. It
can hence be viewed as a regularized Wasserstein distance with
desirable convexity and regularity properties. 
A main advantage of Sinkhorn divergence over Wasserstein
distance lies in its computability by an iterative algorithm
known as Sinkhorn's matrix scaling algorithm \cite{Sinkhorn:64}, where
each iteration involves two matrix-vector products. 
Consequently, Sinkhorn divergence may serve as a feasible alternative to classical
Waaserstein distance, particularly when probability measures are
supported on multi-dimensional spaces in ${\mathbb R}^d$ with $d \ge 2$.

Given a cost function and two $n$-dimensional discrete probability
vectors 
supported on two measurable subsets of the Euclidean space ${\mathbb R}^d$,
it is known that Sinkhorn’s algorithm can compute an $\varepsilon$-approximation of
the original Kantorovich optimal transport problem within ${\mathcal O}(\log n
\, \varepsilon^{-2})$ iterations \cite{Dvurechensky_etal:18}. If the matrix-vector products needed in each
iteration of the algorithm are performed with cost ${\mathcal O}(n^2)$, the overall complexity
of the algorithm for computing Sinkhorn divergence will be ${\mathcal O}(n^2 \, \log n
\, \varepsilon^{-2})$. While theoretically attainable, this quadratic
cost may still prohibit the application of Sinkhorn divergence to
large-scale optimization problems that require many evaluations of
Sinkhorn-driven loss functions.

In the present work, we will develop a hierarchical low-rank strategy
that significantly reduces the quadratic cost of Sinkhorn's algorithm to a
near-linear complexity ${\mathcal O}(n \, \log^3 n \,
\varepsilon^{-2})$. 
We consider a class of cost matrices in optimal
transport that can be decomposed into a $d$-term sum of Kronecker product factors, where each term is asymptotically
smooth. 
Importantly, such class of cost matrices induce a
wide range of optimal transport distances, including the quadratic Wasserstein metric and its corresponding
Sinkhorn divergence. 
We then propose a strategy that takes two steps to reduce the
complexity of each Sinkhorn iteration from ${\mathcal O}(n^2)$ to
${\mathcal O}(n \, \log^2 n)$. 
%
%
We first exploit the block structure of Kronecker product matrices and
turn the original (and possibly high-dimensional) matrix-vector
product problems into several smaller one-dimensional problems. 
The smaller problems will then be computed by the hierarchical matrix technique
\cite{Hackbusch:15}, thanks to their asymptotic smoothness, with a
near-linear complexity. We further present a thorough error-complexity
analysis of the proposed hierarchical low-rank Sinkhorn’s algorithm.  


The rest of the paper is organized as follows. In Section \ref{sec:OT} we review
the basics of Wasserstein and Sinkhorn dissimilarity measures in
optimal transport and include a short discussion of the original Sinkhorn's algorithm. 
We then present and analyze the proposed hierarchical low-rank Sinkhorn’s algorithm for computing Sinkhorn
divergence in Section \ref{sec:ALG}. 
In Section \ref{sec:Numerics} we present a numerical example related
to a signal processing optimization
problem to demonstrate the applicability, accuracy, and
efficiency of the proposed method. 
Finally, in Section \ref{sec:CON} we summarize conclusions and outline future works.

\section{Optimal transport dissimilarity measures}
\label{sec:OT}

In this section, we present the basics of the Kantorovich
formulation of optimal transport and review two related dissimilarity measures; the Wasserstein distance and the Sinkhorn
divergence. We focus on probability vectors, 
rather than working with more general probability measures. 

\subsection{Kantorovich's optimal transport problem and Wasserstein distance}
\label{sec:Wass}

Let ${\mathcal X}$ and ${\mathcal Y}$ be two measurable subsets of the
Euclidean space ${\mathbb R}^d$, with $d \in {\mathbb N}$. 
Let further ${\bf f}$ and ${\bf g}$ be two $n$-dimensional probability
vectors, i.e. two vectors in the probability simplex  $\Sigma_n := \{ {\bf
  f} \in {\mathbb R}_+^n: \, \, {\bf f}^{\top} \mathds{1}_n = 1 \}$,
defined on ${\mathcal X}$ and ${\mathcal Y}$, respectively. Here,
$\mathds{1}_n$ is the $n$-dimensional vector of ones. 
The two probability vectors ${\bf f}$ and ${\bf g}$ are assumed to be
given at two sets of $n$ discrete points $\{ {\bf x}_1, \dotsc, {\bf
  x}_n \} \subset
{\mathcal X}$ and $\{ {\bf y}_1, \dotsc, {\bf y}_n \} \subset
{\mathcal Y}$ in ${\mathbb R}^d$, respectively. Figure
\ref{fig:p_vectors} shows a schematic representation of probability
vectors in $d=1$ and $d=2$ dimensions. 
\begin{figure}[!ht]
\begin{center}
\includegraphics[width=0.3\textwidth]{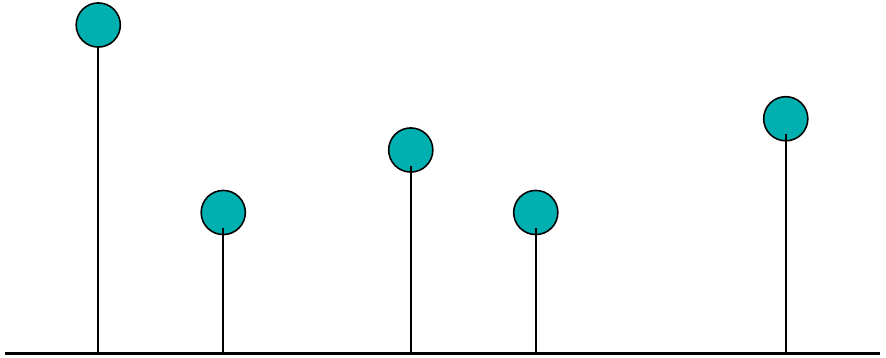}      
\hskip 1cm
\includegraphics[width=0.3\textwidth]{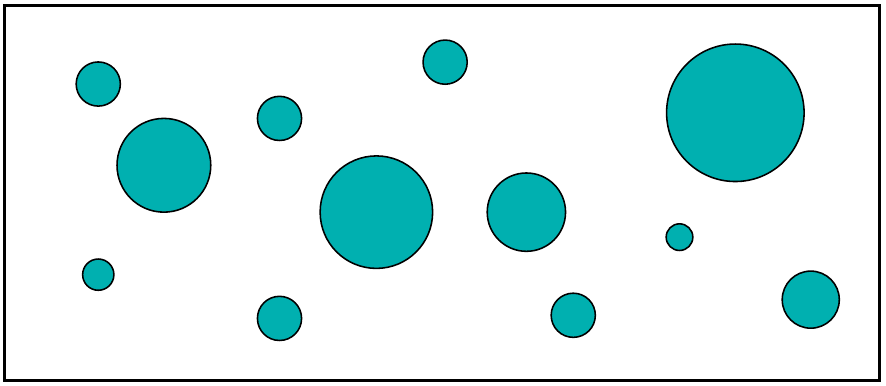}     
\caption{A schematic representation of a probability vector ${\bf f}
  \in \Sigma_n$ given at a set of $n$ discrete points $\{ {\bf x}_1, \dotsc, {\bf
  x}_n \} \subset
{\mathcal X} \subset {\mathbb R}^d$ in one ($d=1$) and two ($d=2$) dimensions. 
Left: a probability vector with $n=5$ components supported in
${\mathbb R}$ and represented by a
  lollipop chart. The hight of each bar is equal to a component ${\bf f}_i$, and the location of each bar corresponds
  to a point ${\bf x}_i \in {\mathcal X} \subset {\mathbb R}$. 
Right: a probability vector with $n=12$ components supported in
${\mathbb R}^2$ and represented by a
  bubble chart, where the area and the center of each bubble correspond to ${\bf
    f}_i$ and ${\bf x}_i
  \in {\mathcal X} \subset {\mathbb R}^2$, respectively.}  
\label{fig:p_vectors}
\end{center}
\end{figure}

We denote by $U({\bf f},{\bf g})$ the transport polytope of ${\bf f}$
and ${\bf g}$, i.e. the set of all nonnegative $n \times n$ matrices
with row and column sums equal to ${\bf f}$ and ${\bf g}$,
respectively, that is,
$$
U({\bf f},{\bf g}) := \{ P \in {\mathbb R}_+^{n \times n}: \, \, P
\mathds{1}_n = {\bf f}, \, P^{\top} \mathds{1}_n = {\bf g} \}.
$$
Each matrix $P = [P_{ij}] \in {\mathbb R}_+^{n \times n}$ in the transport polytope $U({\bf f},{\bf g})$ is a
transport (or coupling) matrix that encodes a transport plan, that is, 
each element $P_{ij}$ of $P$ describes the amount of mass to be
transported from a source point ${\bf x}_i \in {\mathcal X}$ to a
traget point ${\bf y}_j \in {\mathcal Y}$, with $i,j = 1, \dotsc,
n$. The two constraints $P \mathds{1}_n = {\bf f}$ and $P^{\top}
\mathds{1}_n = {\bf g}$ are necessary for a transport plan $P$ to be
admissible; all the mass taken from a point ${\bf x}_i$, i.e. $\sum_{j=1}^n
P_{ij}$, must be equal to the mass at point ${\bf x}_i$, i.e. ${\bf
  f}_i$, and all the mass transported to a point ${\bf y}_j$, i.e. $\sum_{i=1}^n
P_{ij}$, must be equal to the mass at point ${\bf y}_j$, i.e. ${\bf
  g}_j$. We say that ${\bf f}$ and ${\bf g}$ are the marginals of
$P$.

Let further $c: {\mathcal X} \times {\mathcal Y} \rightarrow {\mathbb
  R}_+$ be a non-negative cost function defined on ${\mathcal X}
\times {\mathcal Y}$. For any pair of points $({\bf x},{\bf y}) \in {\mathcal X} \times {\mathcal Y}$, the value $c({\bf x},{\bf y})$ represents the cost of transporting one
unit of mass from a source point ${\bf x} \in {\mathcal X}$ to a
target point ${\bf y} \in {\mathcal Y}$. 
Correspondingly, we consider the cost matrix 
$$
C = [C_{ij}] \in {\mathbb R}_+^{n \times n}, \qquad C_{ij} = c({\bf x}_i, {\bf
  y}_j), \qquad i,j = 1, \dotsc, n.
$$

The optimal transport problem can now be formulated as follows. We first
note that for any given transport plan $P \in U({\bf f},{\bf g})$, the total
transport cost associated with $P$ is given by the Frobenius inner
product $\langle P, C \rangle = \sum_{i,j} P_{ij} C_{ij}$. 
%
%
Kantorovich’s optimal transport problem then aims at minimizing the
total transport cost over all admissible transport plans. It reads
\begin{equation}\label{OT_Kantorovich}
T_{C}({\bf f}, {\bf g}) := \min_{P \in U({\bf f}, {\bf g})} \langle P, C \rangle,
\end{equation}
where $T_C({\bf f}, {\bf g})$ is the optimal total cost of transporting ${\bf f}$ onto
${\bf g}$.


In some applications the main goal of solving Kantorovich’s optimal
transport problem \eqref{OT_Kantorovich} is to find the optimal transport plan, i.e. the
transport matrix $P$ that minimizes the total cost. 
In some other applications the goal is to directly find the optimal cost $T_C$
without the need to explicitly form the optimal transport matrix. An
example of the latter application appears in the construction of loss
functions in optimization problems, where loss functions are built based on metrics induced by the optimal cost $T_C$. An important type of such metrics is Wasserstein metric. 
%

For the optimal cost $T_C$ in \eqref{OT_Kantorovich}
to induce a metric, we need the cost matrix
$C \in {\mathbb R}_+^{n \times n}$ to be a metric matrix belonging to the cone of distance matrices \cite{Villani:09}, i.e. we need $C \in
{\mathcal M}^{n \times n}$, where
$$
{\mathcal M}^{n \times n} := \{ C=[C_{ij}] \in {\mathbb R}_+^{n \times n}:
C_{ij} = 0  \Leftrightarrow i=j, \, \, C_{ij} \le C_{ik} + C_{kj} \}.
$$ 
Throughout this paper, we always assume that the cost matrix $C$ is a
metric matrix: $C \in {\mathcal M}^{n \times n}$. An important class of metric
matrices is driven by cost functions of the form $c({\bf x},{\bf y}) = d({\bf x}, {\bf
  y})^p$, where $p \in [1,\infty)$, and $d$ is a distance function (or
metric). 
In this case the cost matrix reads
\begin{equation}\label{C1}
C = [C_{ij}] \in {\mathbb R}_+^{n \times n}, \qquad C_{ij} = d({\bf
  x}_i, {\bf y}_j)^p, \qquad i,j = 1, \dotsc, n, \qquad p\in
[1,\infty). 
\end{equation}
%
%
%
%
%
%
Intuitively, this choice implies that the cost of transporting one
unit of mass from a source point ${\bf x} \in {\mathcal X}$ to a
target point ${\bf y} \in {\mathcal Y}$ corresponds to the distance
between the two points. 
The optimal cost $T_C({\bf f}, {\bf g})$ in \eqref{OT_Kantorovich}
corresponding to the cost matrix \eqref{C1} induces a metric known as the Wasserstein
distance of order $p$ \cite{Villani:03,Villani:09}, 
\begin{equation}\label{W}
W_p({\bf f}, {\bf g}) = \bigl( T_{C}({\bf f}, {\bf g}) \bigr)^{1/p}.
\end{equation}
%

Due to its desirable convexity and convergence features, Wasserstein
distance is considered as an important measure of dissimilarity
in many applications, ranging from computer vision and machine learning to Seismic and
Bayesian inversion; 
see e.g. \cite{Peyre_Cuturi:19,engquist2016optimal,Wasserstein_Bjorn,Motamed_Appelo:19}.
Unfortunately, the numerical computation of Wasserstein metric is in
general infeasible in high dimensions and is mainly limited to low dimensions; see e.g. \cite{Peyre_Cuturi:19}. 
For instance, in the case $c({\bf x}, {\bf y}) = || {\bf x}
- {\bf y} ||^2$, the quadratic Wasserstein distance between two
continuous probability density functions $f: {\mathcal X}\rightarrow
{\mathbb R}+$ and $g: {\mathcal Y}\rightarrow
{\mathbb R}+$ is given by \cite{Brenier:91}
$$
W_2^2(f, g) = \int_{\mathcal X} | {\bf x} - M({\bf x}) |^2 \, f(\bf x)
\, d{\bf x},
$$
where $M$ is the optimal map. In one dimension, i.e. when ${\mathcal X},
{\mathcal Y} \subset {\mathbb R}$, the optimal map is given by $M=G^{-1} \circ F$,
where $F$ and $G$ are the cumulative distribution functions
corresponding to $f$ and $g$, respectively. The discrete
$W_2$ metric between the corresponding probability vectors ${\bf f} = [f({\bf
  x}_i)]$ and ${\bf g} = [g({\bf y}_i)]$ can then be efficiently approximated by a quadrature rule, 
$$
W_2^2({\bf f}, {\bf g}) \approx \Delta {\bf x} \, \sum_{i=1}^n |{\bf x}_i - G^{-1} \circ F ({\bf x}_i)|^2 {\bf f}_i, \qquad F ({\bf x}_i) =
\sum_{j=1}^i {\bf f}_j,
$$
where $G^{-1}$ may be computed by interpolation and an efficient search
algorithm, such as binary search. The overall cost of computing the discrete
$W_2$ metric in one dimension will then be ${\mathcal O}(n \, \log n)$. 
In multiple dimensions, i.e. when ${\mathcal X}, {\mathcal Y} \subset
{\mathbb R}^d$ with $d \ge 2$, however, the computation of the optimal
map $M$, and hence the $W_2$ metric, may be
either expensive or infeasible. The map is indeed given by the gradient of a convex
function that is the solution to a constrained nonlinear
Monge-Amp{\`e}re equation. We refer to \cite{engquist2016optimal} for more discussion on
the finite difference approximation of the Monge-Amp{\`e}re equation
in two dimensions.

\subsection{Entropy regularized optimal transport and Sinkhorn divergence}
\label{sec:Sink}

One approach to computing Kantorovich's optimal transport
problem \eqref{OT_Kantorovich} is to regularize the original problem
and seek approximate solutions to the regularized problem. 
One possibility is to add an entropic penalty term to the total
transport cost and arrive at the regularized problem
\begin{equation}\label{OT_Regularized}
T_{C}^{\lambda}({\bf f}, {\bf g}) := \min_{P \in U({\bf f}, {\bf g})}
\langle P, C \rangle - \frac{1}{\lambda} H(P),
\end{equation}
where $\lambda >0$ is a regularization parameter, and $H(P)$ is the discrete entropy of transport matrix,
\begin{equation}\label{entropy}
H(P) = - \sum_{i,j} P_{ij} (\log P_{ij} - 1).
\end{equation}
The regularized problem \eqref{OT_Regularized}-\eqref{entropy} with $C
\in {\mathcal M}^{n \times n}$
has a
unique solution, say $P_{\lambda}$ \cite{Cuturi:13}. Indeed, the existence and uniqueness
follows from the boundedness of $U ({\bf f}, {\bf g})$ and the strict convexity of
the negative entropy. 
On the one hand, as $\lambda$ increases, the unique
solution $P_{\lambda}$ converges to the solution with maximum entropy within the set
of all optimal solutions of the original Kantorovich's problem; see Proposition 4.1 in \cite{Peyre_Cuturi:19},
$$
P_{\lambda} \stackrel{\lambda \rightarrow \infty}{\longrightarrow} \underset {P \in  \{ P \in U({\bf f}, {\bf g}), \,
\langle P, C \rangle = T_{C}({\bf f}, {\bf g})\}}{\mathrm{argmin}}  -H(P).
$$  
%
In particular, we have $T_{C}^{\lambda}({\bf f}, {\bf g})
\stackrel{\lambda \rightarrow \infty}{\longrightarrow} T_{C}({\bf f},
{\bf g})$. On the other hand, as $\lambda$ decreases, the unique solution
$P_{\lambda}$ converges to the transport matrix with maximum entropy
between the two marginals ${\bf f}$ and ${\bf g}$, i.e. the outer
product of ${\bf f}$ and ${\bf g}$; see Proposition 4.1 in \cite{Peyre_Cuturi:19},
%
%
$$
P_{\lambda}
\stackrel{\lambda \rightarrow 0}{\longrightarrow} \, {\bf f} \, {\bf
  g}^{\top} = ({\bf f}_i \, {\bf g}_j).
$$ 
It is to be noted that in
this latter case, as $\lambda \rightarrow 0$, the optimal solution $P_{\lambda}$ becomes less
and less sparse (or smoother), i.e. with more and more entries larger than a
prescribed small threshold. 

With $P_{\lambda}$ being the optimal solution to the regularized problem
\eqref{OT_Regularized}-\eqref{entropy} corresponding to the cost
matrix \eqref{C1}, the Sinkhorn divergence of
order $p$ between ${\bf f}$ and ${\bf g}$ is defined as
\begin{equation}\label{S}
S_{p,\lambda}({\bf f}, {\bf g}) = \langle P_{\lambda}, C \rangle^{1/p}.
\end{equation}

Sinkhorn divergence \eqref{S} has several important properties. First, it
satisfies all axioms for a metric with the exception of identity of
indiscernibles (or the coincidence axiom), that is, it satisfies
non-negativity, symmetry, and triangle inequality axioms
\cite{Cuturi:13}. 
Second, it is convex and smooth with respect to the input probability
vectors and can be differentiated using automatic differentiation \cite{Peyre_Cuturi:19}. 
Third, as discussed below, it can be computed by a simple alternating
minimization algorithm, known as Sinkhorn's algorithm, where each
iteration involves two matrix-vector products. 
Consequently, Sinkhorn divergence may serve as a feasible alternative
to classical Waaserstein distance, particularly when probability
measures are supported on multi-dimensional metric spaces. 
%
%
%
%
\begin{proposition} \label{Prop:S-W}
Sinkhorn divergence \eqref{S} is an upper
  bound for Wasserstein distance \eqref{W},
\begin{equation}\label{S-W}
S_{p,\lambda}({\bf f}, {\bf g}) \ge W_p({\bf f}, {\bf g}).
\end{equation}
\end{proposition}
\begin{proof}
Since $P_{\lambda}$ is the solution to
\eqref{OT_Regularized}-\eqref{entropy}, then we have $P_{\lambda} \in U({\bf
  f}, {\bf g})$. By \eqref{OT_Kantorovich} we therefore get
$\langle P_{\lambda}, C \rangle \ge T_C({\bf f}, {\bf g})$, and hence
the inequality \eqref{S-W} follows by \eqref{W} and \eqref{S}. 
\end{proof}


\medskip
\noindent
{\bf Sinkhorn's algorithm.} 
Adding an entropic penalty term to the optimal transport problem
enforces a simple structure on the regularized optimal transport
matrix $P_{\lambda}$. 
Introducing two dual variables $\hat{\bf f} \in {\mathbb R}^n$ and
$\hat{\bf g} \in {\mathbb R}^n$ for the two marginal constraints $P \,
\mathds{1}_n = {\bf f}$ and $P^{\top} \, \mathds{1}_n ={\bf g}$ in $U({\bf f},{\bf g})$, the
Lagrangian of \eqref{OT_Regularized}-\eqref{entropy} reads
$$
{\mathcal L}(P, \hat{\bf f}, \hat{\bf g}) = \langle P, C \rangle - \frac1{\lambda}
H(P) - \hat{\bf f}^{\top} \, (P \, \mathds{1}_n - {\bf f}) - 
\hat{\bf g}^{\top} \, (P^{\top} \, \mathds{1}_n - {\bf g}).
$$
Setting $\partial_{P_{ij}} {\mathcal L}  = C_{ij} + \frac1{\lambda} \log
P_{ij} - \hat{\bf f}_i - \hat{\bf g}_j = 0$, we obtain 
\begin{equation} \label{P-scale1} 
P_{ij} = {\bf u}_i \, Q_{ij} \, {\bf v}_j, \qquad Q_{ij} :=\exp({-\lambda
  \, C_{ij}}), \qquad {\bf u}_i:=\exp({\lambda \, \hat{\bf
    f}_i}), \qquad {\bf v}_j := \exp (\lambda \, \hat{\bf g}_j).
\end{equation}
We can also write \eqref{P-scale1} in matrix factorization form,
\begin{equation} \label{P-scale2} 
P_{\lambda} = U\, Q \, V, \qquad 
U = \text{diag}({\bf u}_1, \dotsc, {\bf u}_n), \qquad 
Q = [Q_{ij}], \qquad 
V=\text{diag}({\bf v}_1, \dotsc, {\bf v}_n).
\end{equation}
A direct consequence of the form \eqref{P-scale1}, or equivalently
\eqref{P-scale2}, is that $P_{\lambda}$ is non-negative,
i.e. $P_{\lambda} \in {\mathbb R}_+^{n \times n}$. Moreover, it implies
that the computation of $P_{\lambda}$ amounts to computing two non-negative
vectors $({\bf u}, {\bf v}) \in {\mathbb R}_+^n \times {\mathbb
  R}_+^n$, known as scaling vectors. These two vectors can be obtained from the two
marginal constraints,
\begin{equation}\label{LSS1}
U \, Q \, V \, \mathds{1}_n = {\bf f}, \qquad V \, Q^{\top} \, U \, \mathds{1}_n = {\bf g}.
\end{equation}
Noting that $U \, \mathds{1}_n = {\bf u}$ and $V \, \mathds{1}_n =
{\bf v}$, we arrive at the following nonlinear equations for $({\bf
  u}, {\bf v})$,
\begin{equation}\label{LSS2}
{\bf u} \odot (Q \, {\bf v}) = {\bf f}, \qquad {\bf v} \odot (Q^{\top} {\bf u}) = {\bf g},
\end{equation}
where $\odot$ denotes the Hadamard (or entrywise) product. Problem
\eqref{LSS1}, or equivalently \eqref{LSS2}, is known as the
non-negative matrix scaling problem (see
e.g. \cite{Kalantari-Khachiyan:96,Nemirovski-Rothblum:99}) and can be solved by an iterative method known as Sinkhorn's algorithm \cite{Sinkhorn:64},
\begin{equation}\label{Alg-Sinkhorn}
{\bf u}^{(i)} = {\bf f} \, \oslash \, (Q \, {\bf v}^{(i-1)}), \qquad 
{\bf v}^{(i)} = {\bf g} \, \oslash \, (Q^{\top} \, {\bf u}^{(i)}),
\qquad i=1, \dotsc, K.
\end{equation}
Here, $\oslash$ denotes the Hadamard (or entrywise) division. 
We note that in practice we will need to use a stopping criterion. One
such criterion may be defined by monitoring the difference between the
true marginals $({\bf f}, {\bf g})$ and the marginals of the most
updated solutions. Precisely, given a small tolerance
$\varepsilon_{\text{S}} >0$, we continue Sinkhorn iterations $i=1, 2, \dotsc$ until we obtain
\begin{equation}\label{stop_criterion}
\max \{ || {\bf u}^{(i)} \odot (Q \, {\bf v}^{(i)} )  - {\bf f} ||_{\infty} , \,  ||
{\bf v}^{(i)}  \odot (Q^{\top} {\bf u}^{(i)} ) - {\bf g} ||_{\infty}
\} \le \varepsilon_{\text{S}}.
\end{equation}
After computing the scaling vectors $({\bf u}, {\bf v})$, the $p$-Sinkhorn
divergence can be computed as
\begin{equation}\label{Alg-Sinkhorn2}
S_{p, \lambda} = ( {\bf u}^{\top} \hat{Q} \, {\bf v} )^{1/p},
\qquad \hat{Q} :=  Q \odot C.
\end{equation}
%
%
%
We refer to Remark 4.5 of \cite{Peyre_Cuturi:19} for a historic perspective of
the matrix scaling problem \eqref{LSS2} and the iterative algorithm \eqref{Alg-Sinkhorn}. 

\medskip
\noindent
{\bf Convergence and complexity.} 
Suppose that our goal is to compute an $\varepsilon$-approximation $P^{\ast}$ of the optimal transport plan,
$$
{\text{find}} \, \, P^{\ast} \in U({\bf f},{\bf g}) \ \ \text{s.t.} \ \
\langle P^{\ast}, C \rangle \le \min_{P \in U({\bf f},{\bf g})}
\langle P, C \rangle + \varepsilon,
$$
where $\varepsilon >0$ is a positive small tolerance. 
This can be achieved within $K = {\mathcal O}(|| C ||_{\infty}^2\log n
\, \varepsilon^{-2})$ Sinkhorn iterations if we set $\lambda = 4 \, 
\varepsilon^{-1}  
\log n$ \cite{Dvurechensky_etal:18}. 
See also \cite{Altschuler_etal:17} for a less sharp bound on the number of iterations. 
The bounds presented in \cite{Dvurechensky_etal:18,Altschuler_etal:17} suggest that the number of iterations depends weakly on $n$. 
The main computational bottleneck of Sinkhorn's algorithm is however
the two vector-matrix multiplications against kernels $Q$ and
$Q^{\top}$ needed in each iteration with complexity ${\mathcal O}(n^2)$ if implemented naively. 
Overall, Sinkhorn's algorithm can compute an $\varepsilon$-approximate solution of
the unregularized OT problem in ${\mathcal O}(K  n^2) = {\mathcal
  O}(n^2\log n)$ operations. We also refer to Section 4.3 in \cite{Peyre_Cuturi:19} for a
few strategies that may improve this complexity. In particular, the
authors in \cite{Solomon_etal:15} exploited possible separability of
cost metrics (closely related to assumption A1 in Section
\ref{sec:assumptions} of the present work) to reduce the quadratic
complexity to
${\mathcal O}(n^{1 + 1/d} \, \log n)$; see also Remark 4.17 in
\cite{Peyre_Cuturi:19}. 
It is to be noted that although this improved cost is no longer quadratic (except when $d=1$), the resulting algorithm will still be a polynomial time
algorithm whose complexity includes a term $n^{\gamma}$ with
$\gamma = 1+1/d>1$. 
Our goal here is to exploit other possible structures in the kernels
and further reduce the cost to a log-linear cost, thereby delivering
an algorithm that runs faster than any polynomial time algorithm.

\begin{remark} \label{remark:lambda}
(Selection of $\lambda$) On the one hand, the
  regularization parameter $\lambda$ needs to be large enough for the
  regularized OT problem to be close to the original OT problem. This
  is indeed reflected in the choice $\lambda = 4 \, \varepsilon^{-1}
  \log n$ that enables achieving an $\varepsilon$-approximation of the
  optimal transport plan: the smaller the tolerance
  $\varepsilon$, the larger $\lambda$. On the other hands, the
  convergence of Sinkhorn's algorithm deteriorates as $\lambda \rightarrow \infty$;
  see e.g. \cite{FranklinLorenz:89,Knight:08}. Moreover, as $\lambda
  \rightarrow \infty$, more and more entries of the kernel $Q$ (and
  hence entries of $Q \, {\bf v}$ and $Q^{\top} {\bf u}$) may
  become ``essentially'' zero. In this case, Sinkhorn's algorithm may become numerically
  unstable due to division by zero. The selection of the regularization parameter
  $\lambda$ should therefore be based on a trade-off between
  accuracy, computational cost, and robustness. 
\end{remark}

%
%
%
%
%
%
%
%
%

\section{A fast algorithm for computing Sinkhorn divergence}
\label{sec:ALG}

We consider an important class of cost matrices for which
the complexity of each Sinkhorn iteration can be significantly reduced
from ${\mathcal O}(n^2)$ to ${\mathcal O}(n \, \log^2 n)$. 
Precisely, we consider cost matrices that can be decomposed into a
$d$-term sum of Kronecker product factors, where each term is asymptotically
smooth; see Section \ref{sec:assumptions}. 
The new Sinkhorn's algoritm after $K = {\mathcal O}(\log n)$ iterations
will then have a near-linear complexity ${\mathcal O}(n \, \log^3 n)$;
see Section \ref{sec:H-Sinkhorn} and Section \ref{sec:algorithm}. 
Importantly, such class of cost matrices induce a
wide range of optimal transport distances, including the quadratic Wasserstein metric and its corresponding
Sinkhorn divergence.

\subsection{Main assumptions}
\label{sec:assumptions}

To make the assumptions precise, we first introduce a special Kronecker sum of matrices. 
%
\begin{definition}\label{K_sum}
The ``all-ones Kronecker sum'' of two matrices $A \in {\mathbb R}^{p \times p}$ and $B \in
{\mathbb R}^{q \times q}$, denoted by $A \oplus B$, is defined as
$$
A \oplus B := A \otimes J_q + J_p \otimes B, 
$$
where $J_p$ is the all-ones matrix of size $p \times p$, and $\otimes$
denotes the standard Kronecker product. 
\end{definition}
It is to be noted that the all-ones Kronecker sum is different from the common
Kronecker sum of two matrices in which identity matrices $I_p$ and $I_q$
replace all-ones matrices $J_p$ and $J_q$. This new operation facilitates working with elementwise matrix operations. 
%
%
We also note that the two-term sum in Definition \ref{K_sum} can be recursively
extended to define all-ones Kronecker sums with arbitrary number
of terms, thanks to the following associative property inherited from matrix algebra,
$$
A_1 \oplus A_2 \oplus A_3 := (A_1 \oplus A_2) \oplus A_3 = A_1 \oplus (A_2 \oplus A_3).
$$

We next consider kernel matrices generated by kernel functions with a special type of regularity, known as ``asymptotic smoothness'' \cite{Hackbusch:15}.
\begin{definition}\label{Asymptotically_Smooth}
A kernel function $\kappa: {\mathbb R}
\times {\mathbb R} \rightarrow {\mathbb R}$ on the real line 
is said to be ``asymptotically smooth'' if there exist constants $c_0, \alpha, \beta$ such that for some $s \in {\mathbb R}$ the following holds,
\begin{equation}\label{Asym_Smooth_kernel}
| \partial_x^m \kappa(x,y) | \le c_0 \, m! \, \alpha^m \, m^{\beta} \,
|x-y|^{-m -s},  \quad (x,y)\in {\mathbb R}^2, \ \ \ x \neq y, \ \ \
m\in {\mathbb N}.
\end{equation}
A kernel matrix $A = [\kappa(x_i,y_j)] \in {\mathbb R}^{q
  \times q}$, generated by an asymptotically smooth kernel function
$\kappa$, is called an asymptotically smooth kernel matrix.

\end{definition}

We now make the following two assumptions.
\begin{itemize}
\item[A1.] The cost matrix $C \in {\mathcal M}^{n \times n}$ can
  be decomposed into a $d$-term all-ones Kronecker
  sum
\begin{equation}\label{A1}
C = C_d \oplus \dotsi \oplus C_1, \qquad C_k \in {\mathcal M}^{n_k
  \times n_k}, \ \ \  k=1, \dotsc, d, \ \ \ n= \prod_{k=1}^d n_{k},
\end{equation}
where each matrix $C_k \in {\mathcal M}^{n_k \times n_k}$ is generated
by a metric $c_k:{\mathbb R} \times {\mathbb R} \rightarrow {\mathbb
  R}_+$ on the real line.
%

\item[A2.] The kernel matices $Q^{(k)} = \exp[- \lambda C^{(k)}]$ and $\hat{Q}^{(k)}  = C^{(k)} \odot
  Q^{(k)}$, with $k=1, \dotsc, d$, generated by the kernels $\kappa_k(x,y) = \exp(- \lambda \, c_k(x,y))$
  and $\hat{\kappa}_k(x,y) = c_k(x,y) \, \kappa_k(x,y)$, where
  $(x,y)\in {\mathbb R}^2$, are asymptotically smooth. 
\end{itemize}

It is to be noted that by assumption A2 all $2 \, d$ kernels $\kappa_k$ and $\hat{\kappa}_k$, with $k=1, \dotsc, d$,
  satisfy similar estimates of form \eqref{Asym_Smooth_kernel} with possibly
  different sets of constants $(c_0, \alpha, \beta, s)$.

\medskip
\noindent
{\bf $L^p$ cost functions.} An important class of cost functions for which 
assumptions A1-A2 hold includes the $L^p$ cost functions, given by 
\begin{equation}\label{ExC}
c({\bf x}, {\bf y}) = || {\bf x} - {\bf y} ||_p^p = \sum_{k=1}^d | x^{(k)} - y^{(k)} |^p, \quad
 {\bf x}=(x^{(1)}, \dotsc, x^{(d)})  \in {\mathbb R}^d, \ \ \ {\bf y}=(y^{(1)}, \dotsc, y^{(d)}) \in {\mathbb R}^d.
\end{equation}
Here, $||\cdot||_p$ denotes the
$L^p$-norm in ${\mathbb R}^d$, with $p\in [1,\infty)$. 
%
Importantly, such cost functions take the form of \eqref{C1} with $d({\bf x}, {\bf
  y}) = || {\bf x} - {\bf y} ||_p$, for which $p$-Wasserstein metric
and $p$-Sinkhorn divergence are well defined. In particular, the case
$p=2$ corresponds to the quadratic Wasserstein metric. 
%
%
%
%
%
%
%
%

We will first show that assumption A1 holds for $L^p$ cost functions. Let
the two sets of discrete points $\{ {\bf x}_i \}_{i=1}^n \in {\mathcal
  X}$ and $\{ {\bf y}_j \}_{j=1}^n \in
{\mathcal Y}$, at which the two probability vectors
$({\bf f}, {\bf g})$ are defined, 
be given on two regular grids in ${\mathbb
  R}^d$. We note that if the original probability vectors are given
on irregular grids, we may find their values on regular grids by
interpolation with a cost ${\mathcal O}(n)$ that would not change the
near-linearity of our overall target cost. 
%
Set $I_n := \{ 1, \dotsc, n \}$. Then, to each pair of indices $i \in I_n$ and $j \in I_n$ of the cost matrix $C =
[c({\bf x}_i, {\bf y}_j)]$, we can assign a pair of $d$-tuples $(i_1, \dotsc,i_d)$ and $(j_1,
\dotsc, j_d)$ in the Cartesian product of $d$ finite sets $I_{n_1} \times \dotsc \times I_{n_d}$, where $n= \prod_{k=1}^d
n_{k}$. Consequently, the additive representation of $c$ in \eqref{ExC} will imply
\eqref{A1}. 

It is also easy to see that assumption A2 holds for $L^p$ cost
functions, with corresponding kernels,
$$
\kappa(x, y) = \exp (- \lambda \, |x-y|^p), \qquad \hat{\kappa}(x ,y) =
|x-y|^p \, \exp (- \lambda \, |x-y|^p), \qquad (x,y) \in {\mathbb
  R}^2.
$$
Indeed, these kernels are asymptotically smooth due to the
exponential decay of $\exp (- \lambda \, |x -y|^p)$ as $|x-y|$
increases. For instance, in the particular case when $p=2$, the
estimate \eqref{Asym_Smooth_kernel} holds for both $\kappa$ and $\hat{\kappa}$ with
$c_0=1$, $\alpha=2$, $\beta = 0$, and $s=0$. 
%
%
To show this, we note that
$$
\partial_x^m \exp (- \lambda \, (x-y)^2)  = (- \sqrt{\lambda})^m
H_m(\sqrt{\lambda}(x-y)) \, \exp (- \lambda \, (x-y)^2),
$$
where $H_m(z)$ is the Hermite polynomial of degree $m$, satisfying the
following inequality \cite{Indritz:61},
$$
H_m(z) \le  (2^m \, m!)^{1/2} \exp(z^2/2), \qquad z \in {\mathbb R}.
$$
Setting $z := \sqrt{\lambda}(x-y) \neq 0$, and noting that $ (2^m \,
m!)^{1/2} \le  2^m \, m!$ for $m \in {\mathbb N}$, we obtain
$$
|\partial_x^m \exp (- z^2)|  = (\sqrt{\lambda})^m
|H_m(z)| \, \exp (- z^2) \le (2 \sqrt{\lambda})^m\, m! \, \exp(-z^2/2)
\le  (2 \sqrt{\lambda})^m\, m! \, |z|^{-m}.
$$
Hence we arrive at
$$
|\partial_x^m \exp (-  \lambda \, (x-y)^2)|  \le  2^m\, m! \, |x-y|^{-m}.
$$
Similarly, we can show that the same estimate holds for $|\partial_x^m
(x-y)^2 \exp (-  \lambda \, (x-y)^2)|$. 

%

\subsection{Hierarchical low-rank approximation of Sinkhorn iterations}
\label{sec:H-Sinkhorn}

As noted in Section \ref{sec:Sink},  the main computational
bottleneck of Sinkhorn's algorithm is the matrix-vector
multiplications by kernels $Q$ and $Q^{\top}$ in each
iteration \eqref{Alg-Sinkhorn} and by kernel $\hat{Q}$ in \eqref{Alg-Sinkhorn2}. 
We present a strategy that reduces the complexity of each Sinkhorn
iteration to ${\mathcal O}(n \, \log^2
n)$, enabling Sinkhorn's algorithm to achieve a
near-linear overall complexity ${\mathcal O}(n \, \log^3 n)$. 
The proposed strategy takes two steps, following two observations made based on assumptions A1 and
A2.



\subsubsection{Step 1: decomposition into 1D problems}
\label{sec:part1}

A direct consequence of assumption A1 is that both kernel matrices $Q$
and $\hat{Q}$ have separable multiplicative structures. 

\begin{proposition} \label{Prop_separable}
Under assumption A1 on the cost matrix $C$, the matrices
$Q=\exp[-\lambda C]$ and $\hat{Q}=Q \odot C$ will have Kronecker product structures,
\begin{equation}\label{Q_form}
Q = \exp [-\lambda C] = Q^{(d)} \otimes \dotsi \otimes Q^{(1)}, \qquad
Q^{(k)} = \exp[- \lambda C^{(k)}], \ \ \ k= 1, \dotsc, d,
\end{equation}
\begin{equation}\label{Qhat_form}
\hat{Q}=Q \odot C = \sum_{k=1}^d A_k^{(d)} \otimes
\dotsi \otimes A_k^{(1)}, 
\qquad 
A_k^{(m)} =
\left\{ \begin{array}{ll}
\hat{Q}^{(k)} := C^{(m)} \odot Q^{(m)}, &  k = m \\
Q^{(m)}, & k \neq m 
\end{array} \right..
\end{equation}
\end{proposition}
\begin{proof} 
The first expression \eqref{Q_form} follows from Lemma \ref{lem1} in
the appendix and \eqref{A1}. The
second expression \eqref{Qhat_form} follows from Lemma \ref{lem2} in
the appendix and \eqref{A1} and \eqref{Q_form}.
\end{proof}

Exploiting the block structure of the Kronecker product matrices $Q$
and $\hat{Q}$, one can compute the matrix-vector products $Q {\bf
  w}$, $Q^{\top} {\bf w}$, and $\hat{Q} {\bf w}$ without
explicitly forming $Q$ and $\hat{Q}$. 
This can be achieved using the vector operator $\text{vec}(.)$ that
transforms a matrix into a vector by stacking its columns beneath one
another. 
We denote by ${\bf a} =\text{vec}(A)$ the vectorization of matrix $A
\in {\mathbb R}^{p \times q }$
formed by stacking its columns into a single column vector ${\bf a}
\in {\mathbb R}^{p q  \times 1 }$.

\begin{proposition} \label{Prop_Aw}
Let $A = A^{(d)} \otimes \dotsi \otimes A^{(1)} \in {\mathbb R}^{n \times n}$, where $A^{(k)} \in
{\mathbb R}^{n_k \times n_k}$ and $n= \prod_{k=1}^d n_{k}$. Then, the
matrix-vector multiplication $A {\bf w}$, where ${\bf w} \in {\mathbb
  R}^{n \times 1}$, amounts to 
$\tilde{n}_k= \prod_{\substack{i=1, \, i \neq k}}^d n_i$ 
matrix-vector multiplications $A^{(k)} {\bf z}$ for $\tilde{n}_k$ different vectors ${\bf z} \in {\mathbb R}^{n_k \times
  1}$, with $k=1,\dotsc, d$. 
\end{proposition}
\begin{proof} 
The case $d=1$ is trivial, noting that $\tilde{n}_1 = 1$. 
Consider the case $d=2$. Then we have $A = A^{(2)} \otimes
A^{(1)} \in {\mathbb R}^{n \times n}$, with $A^{(1)} \in {\mathbb
  R}^{n_1 \times n_1}$, $A^{(2)} \in {\mathbb
  R}^{n_2 \times n_2}$, and $n = n_1  n_2$. 
Let $W \in {\mathbb R}^{n_1 \times n_2}$ be the matrix whose
vectorization is the vector ${\bf w} \in {\mathbb R}^{n_1  n_2
  \times 1 }$, i.e. ${\bf w} =\text{vec}(W)$. 
Then, we have
$$
A {\bf w} = (A^{(2)} \otimes A^{(1)}) \, {\text{vec}}(W) =
\text{vec}(A^{(1)} W A^{(2) \top}).
$$
Hence, computing $A {\bf w}$ amounts to computing $A^{(1)} W A^{(2)
  \top}$. 
This can be done in two steps. 
We first compute $G:=W  A^{(2) \top}$,
$$
G:= W  A^{(2) \top} = 
\left(\begin{array}{c}
  {\bf w}_1^{\top}\\
  \vdots \\
  {\bf w}_{n_1}^{\top}
\end{array}\right) A^{(2) \top} = 
\left(\begin{array}{c}
  {\bf w}_1^{\top} A^{(2) \top}\\
  \vdots \\
  {\bf w}_{n_1}^{\top} A^{(2) \top}
\end{array}\right) =
\left(\begin{array}{c}
 (A^{(2)}  {\bf w}_1)^{\top}\\
  \vdots \\
  (A^{(2)}  {\bf w}_{n_1})^{\top}
\end{array}\right),
$$
where ${\bf w}_k \in {\mathbb R}^{n_2 \times 1}$ is the $k$-th row of
matrix $W$ in the form of a column vector. 
For this we need to perform $n_1$
matrix-vector multiplies $A^{(2)}  {\bf w}_k$, with $k=1, \dotsc, n_1$. 
%
We then compute $A^{(1)} G$,
$$
A^{(1)} G = A^{(1)} [ {\bf g}_1 \dotsi {\bf g}_{n_2} ] =
[A^{(1)} {\bf g}_1 \dotsi A^{(1)} {\bf g}_{n_2} ], 
$$
where ${\bf g}_k \in {\mathbb R}^{n_1 \times 1}$ is the $k$-th column of
matrix $G$. We therefore need to perform $n_2$ matrix-vector
multiplies $A^{(1)} {\bf g}_k$, with $k=1, \dotsc, n_2$. 
%
%
%
%
%
%
Overall, we need $n_1$ matrix-vector multiplies $A^{(2)}  {\bf z}$
for $n_1$ different vectors ${\bf z}$, and $n_2$ matrix-vector
multiplies $A^{(1)} {\bf z}$ for $n_2$ different vectors ${\bf z}$, as
desired. 
This can be recursively extended to any dimension $d \ge 3$. Setting $\tilde{A}:=A^{(d-1)} \otimes \dotsi \otimes A^{(1)}$, we can write
\begin{equation}\label{Aw}
A {\bf w} = (A^{(d)} \otimes A^{(d-1)} \otimes \dotsi \otimes A^{(1)})  {\bf w} 
= (A^{(d)} \otimes \tilde{A})  \, {\text{vec}}(W) 
= \text{vec}(\tilde{A} W A^{(d) \top}).
\end{equation}
This can again be done in two steps. 
We first compute $G:=W  A^{(d) \top}$, where $W \in {\mathbb
  R}^{\tilde{n} \times n_d}$ with $\tilde{n} = \prod_{k=1}^{d-1}n_k$,
$$
G:= W  A^{(d) \top} = 
\left(\begin{array}{c}
  {\bf w}_1^{\top}\\
  \vdots \\
  {\bf w}_{\tilde{n}}^{\top}
\end{array}\right) A^{(d) \top} = 
\left(\begin{array}{c}
  {\bf w}_1^{\top} A^{(d) \top}\\
  \vdots \\
  {\bf w}_{\tilde{n}}^{\top} A^{(d) \top}
\end{array}\right) =
\left(\begin{array}{c}
 (A^{(d)}  {\bf w}_1)^{\top}\\
  \vdots \\
  (A^{(d)}  {\bf w}_{\tilde{n}})^{\top}
\end{array}\right),
$$
where ${\bf w}_k \in {\mathbb R}^{n_d \times 1}$ is the $k$-th row of
matrix $W$ in the form of a column vector. 
To compute $G$ we need to perform $\tilde{n}$
matrix-vector multiplies $A^{(d)}  {\bf w}_k$, with $k=1, \dotsc, \tilde{n}$
%
%
We then compute $\tilde{A} G$,
$$
\tilde{A} G = \tilde{A} [ {\bf g}_1 \dotsi {\bf g}_{n_d} ] =
[ \tilde{A} {\bf g}_1 \dotsi \tilde{A} {\bf g}_{n_d} ], 
$$
where ${\bf g}_k \in {\mathbb R}^{\tilde{n} \times 1}$ is the $k$-th column of
matrix $G$. 
We therefore need to perform $n_d$ matrix-vector multiplies $\tilde{A}
{\bf g}_k$, with $k=1, \dotsc, n_d$. Noting that each multiply $\tilde{A}
{\bf g}_k$ is of the same form as the multiply \eqref{Aw} but in $d-1$ dimensions, the
proposition follows easily by induction. 
\end{proof}

\medskip
\noindent
{\bf Reduction in complexity due to step 1.} 
Proposition \ref{Prop_separable} and Proposition \ref{Prop_Aw} imply
\begin{equation}\label{cost_step1}
\text{cost}(Q {\bf w}) = \sum_{k=1}^d \tilde{n}_k \,
\text{cost}(Q^{(k)} {\bf z}), 
\qquad
\text{cost}(\hat{Q} {\bf w}) = d \, \text{cost}(Q {\bf w}).
\end{equation}
Indeed, the original $d$-dimensional problems $Q {\bf
  w}$ and $\hat{Q} {\bf w}$ turn into several smaller one-dimensional
problems of either $Q^{(k)} {\bf z}$ or $\hat{Q}^{(k)} {\bf
  z}$ form. This already amounts to a significant gain in efficiency. For instance, if we compute each of these smaller problems by performing ${\mathcal
  O}(n_k^2)$ floating-point arithmetic operations, then the complexity
of $Q {\bf w}$ will become 
$$
\text{cost}(Q {\bf w}) =\sum_{k=1}^d \tilde{n}_k {\mathcal
  O}(n_k^2) = {\mathcal O}(n \sum_{k=1}^d n_k).
$$
And in the particular case when $n_1 =
\dotsi = n_d = n^{1/d}$, we will have 
$$
\text{cost}(Q {\bf w}) = d \, \mathcal {O}(n^{1+1/d}).
$$ 

It is to be noted that this improvement in the efficiency (for $d \ge 2$) was also
achieved in \cite{Solomon_etal:15}. 
Although this improved cost is no longer quadratic (except in 1D
problems when $d=1$), the resulting algorithm will still be a polynomial time
algorithm whose complexity includes a term $n^{\gamma}$ with
$\gamma>1$. For instance for 2D and 3D problems (considered
in \cite{Solomon_etal:15}), the overall complexity of Sinkhorn's algorithm after $K = \mathcal
{O}(\log n)$ iterations would be $\mathcal {O}(n^{3/2} \, \log n)$ and
$\mathcal {O}(n^{4/3} \, \log n)$, respectively. 
In what follows, we will further apply the technique of hierarchical matrices to the
smaller one-dimensional problems, thereby delivering an algorithm that
achieves log-linear
complexity and hence runs faster than any
polynomial time algorithm.

\subsubsection{Step 2: approximation of 1D problems by hierarchical matrices}
\label{sec:part2}
We will next present a fast method for computing the smaller 1D problems $Q^{(k)} {\bf z}$ and $\hat{Q}^{(k)} {\bf z}$
obtained in step 1. 
We will show that, thanks to
assumption A2, the hierarchical matrix technique
\cite{Hackbusch:15} can be employed to compute each one-dimensional
problem with a near-linear complexity ${\mathcal O}(n_k \, \log^2
n_k)$.



Since by assumption A2 the kernels $\kappa_k$ and $\hat{\kappa}_k$ that generate $Q^{(k)}$
and $\hat{Q}^{(k)}$, with $k=1, \dotsc, d$, are all asymptotically
smooth and satisfy similar estimates of form
\eqref{Asym_Smooth_kernel}, the numerical treatment and complexity of
all $2 \, d$ one-dimensional problems will be the same. Hence, in what
follows, we will drop the dependence on $k$
and consider a general one-dimensional problem $A {\bf
  z}$, where $A \in {\mathbb R}^{n_k \times n_k}$ is a kernel matrix
generated by a one-dimensional kernel $\kappa: {\mathbb R} \times
{\mathbb R} \rightarrow {\mathbb R}$ satisfying
\eqref{Asym_Smooth_kernel}. 
Precisely, 
let $I_{n_k} := \{ 1, \dotsc, n_k \}$ be an index set with
cardinality $\# I_{n_k}= n_k$. Consider the kernel matrix 
$$
A := [\kappa(x_i,y_j)] \in {\mathbb R}^{n_k \times n_k}, \qquad i, j
\in I_{n_k}, 
$$
uniquely determined by two sets of
discrete points $\{ x_i \}_{i \in I_{n_k}} \in {\mathbb R}$ and $\{ y_j \}_{j \in
  I_{n_k}} \in {\mathbb R}$ and an asymptotically smooth kernel $\kappa(x,y)$, with
$(x,y) \in {\mathbb R}^2$, satisfying \eqref{Asym_Smooth_kernel}. 
Let $A\arrowvert_{\tau \times \sigma}$ denote a matrix block of the square
kernel matrix $A$, characterized by two index sets $\tau \subset I_{n_k}$ and $\sigma \subset I_{n_k}$. We define the diameter of $\tau$ and the distance between $\tau$ and $\sigma$ by
$$
\text{diam}(\tau) = |\max_{i \in {\tau}} x_i - \min_{i \in {\tau}} x_i|,
\qquad  
 \text{dist}(\tau, \sigma) = \min_{i \in \tau, \, j \in \sigma} |x_i-y_j|.
$$
The two sets of discrete points and the position of the matrix block
is illustrated in Figure \ref{fig_block_matrix}. Note that when the
cardinality of the index sets $\tau$ and $\sigma$ are equal ($\#\tau =
\#\sigma$), the matrix block $A\arrowvert_{\tau \times \sigma}$ will be
a square matrix. 
%
%
%
\begin{figure}[!h]
\begin{center}
\includegraphics[width=0.7\linewidth]{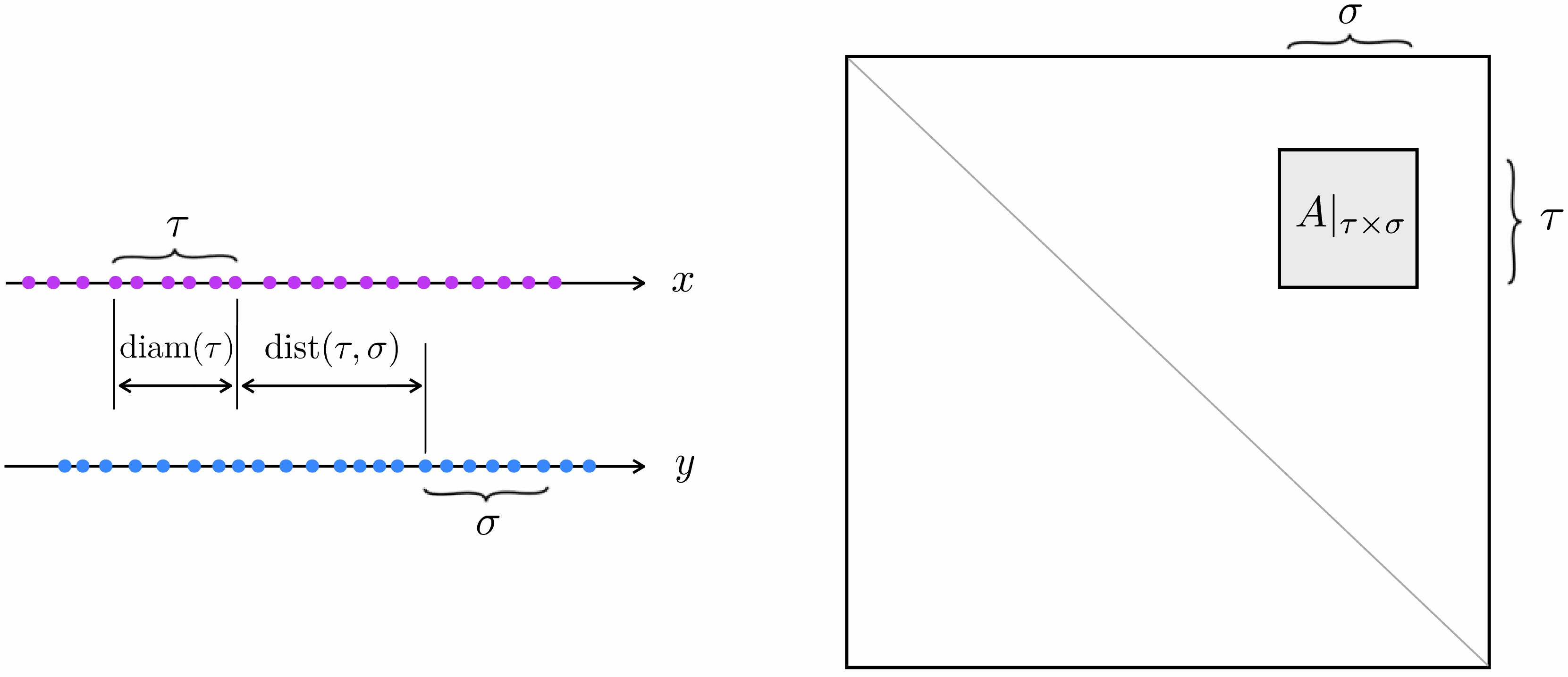}     
\caption{Left: the set of discrete points $\{ x_i \}_{i \in I_{n_k}}$
  (purple circles) and
$\{ y_j \}_{j \in I_{n_k}}$ (blue circles) and two index sets $\tau \subset I_{n_k}$
and $\sigma \subset I_{n_k}$. Right: position of the matrix block
$A\arrowvert_{\tau \times \sigma}$ in the kernel matrix $A=
[\kappa(x_i,y_j)]$ corresponding to the two index sets $\tau$ and $\sigma$.}
\label{fig_block_matrix}
\end{center}
\end{figure}

The hierarchical matrix technique for computing $A {\bf z}$ consists
if three parts (see details below): 
\begin{itemize}
\item[I.] construction of a hierarchical block structure that
  partitions $A$ into low-rank (admissible) and full-rank (inadmissible)
  blocks; 

\item[II.] low-rank approximation of admissible blocks by separable expansions; 

\item[III.] assembling the output $A {\bf z}$ through a loop over all admissible and
  inadmissible blocks.
\end{itemize}

\medskip
\noindent
{\bf I. Hierarchical block structure.} We employ a quadtree structure,
i.e. a tree structure in which each node/vertex (square block) has either
four or no children (square sub-blocks). Starting with the original matrix $A$ as
the tree's root node with $\tau=\sigma= I_{n_k}$, we recursively split each square block $A\arrowvert_{\tau \times \sigma}$ into
four square sub-blocks provided $\tau$ and $\sigma$ do not satisfy the following ``admissibility''
condition,
\begin{equation}\label{admissibility}
\eta := \frac{\text{diam}(\tau)}{\text{dist}(\tau, \sigma)} \le \eta_0.
\end{equation}
Here, $\eta_0>0$ is a fixed number to be determined later; see Section
\ref{sec:accuracy}. 
If the above admissibility condition holds, we do not split the block
$A\arrowvert_{\tau \times \sigma}$ and label it as an admissible leaf. 
We continue this process until the
size of inadmissible blocks reaches a minimum size $n_{\text{min}} \times n_{\text{min}}$. 
This will impose a condition on the cardinality of $\tau$:
$$
\# \tau \ge n_{\text{min}}.
$$
Indeed, we want all sub-block leaves of the tree not to be smaller
than a minimum size for the low-rank approximation of sub-blocks to be
favorable compared to a full-rank computation in terms of complexity. 
It is suggested that we take $n_{\text{min}} =32$ in order to avoid the overhead caused by
recursions; see \cite{Hackbusch:15}. 
%
Eventually, we
collect all leaves of the tree and split them into two partitions: 
\begin{itemize}
\item $P_{\text{admissible}}$: a partition of admissible blocks

\item $P_{\text{inadmissible}}$: a partition of inadmissible blocks
\end{itemize}
Note that by this construction, all blocks in $P_{\text{inadmissible}}$
will be of size $n_{\text{min}} \times n_{\text{min}}$ and are hence
small. Such blocks will be treated as full-rank matrices, while the
admissible blocks will be treated as low-rank matrices. Figure
\ref{fig_partition} shows a schematic representation of the
hierarchical block structure obtained by a quadtree of depth five, that is,
five recursive levels of splitting is performed to obtain the block
structure. Gray blocks are admissible, while red blocks
are inadmissible.
\begin{figure}[!h]
\begin{center}
\includegraphics[width=0.5\linewidth]{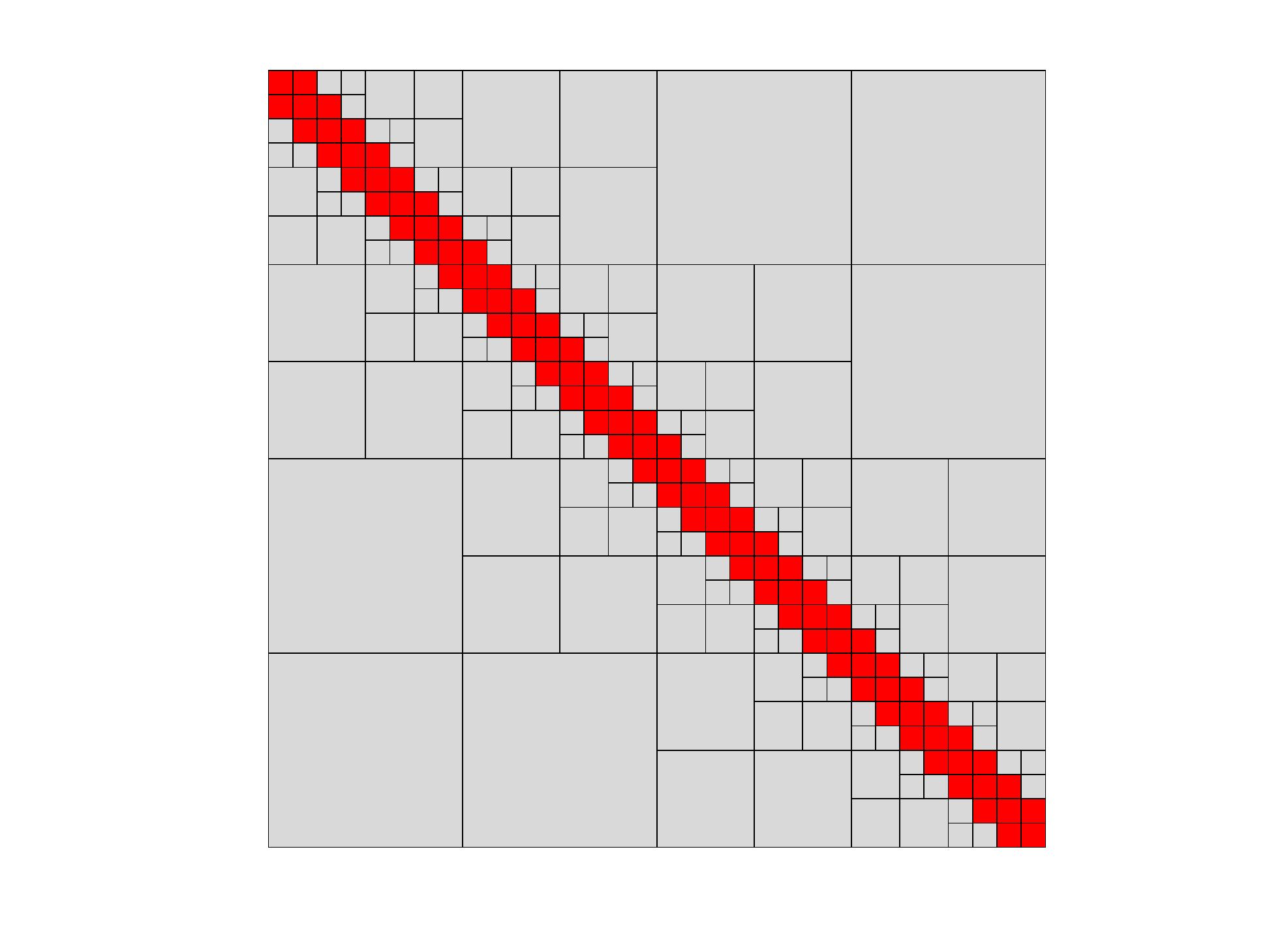}     
\caption{A schematic representation of the hierarchical block
  structure based on a quadtree of depth five. Matrix blocks are
  partitioned into two admissible $P_{\text{admissible}}$ (gray) and
  inadmissible $P_{\text{inadmissible}}$ (red) sets.}
\label{fig_partition}
\end{center}
\end{figure}
%


\medskip
\noindent
{\bf II. Low-rank approximation by separable expansions.} 
Locally on each admissible block $A\arrowvert_{\tau \times \sigma} \in
{\mathbb R}^{b \times b}$, with $b = \# \tau = \# \sigma$, we approximate the kernel $\kappa$ by a separable
expansion, resulting in a low-rank approximation of the
block. Precisely, we approximate $\kappa(x,y)$, where $x \in [\min_{i
  \in {\tau}} x_i , \max_{i \in {\tau}} x_i]$ and $y \in [\min_{j
  \in {\sigma}} y_j , \max_{j \in {\sigma}} y_j]$, using barycentric Lagrange
interpolation in $x$ on a set of $r \in {\mathbb N}$ Chebyshev points $\{ \hat{x}_k
\}_{k=1}^r \in [\min_{i \in {\tau}} x_i , \max_{i \in {\tau}} x_i]$,
where $r \le b$ is a number to be determined; see Section \ref{sec:accuracy}. We write
$$
\kappa(x,y) \approx \kappa^{(r)}(x,y) := \ell(x) \sum_{k=1}^r
\frac{w_k}{x - \hat{x}_k} \,
\kappa(\hat{x}_k,y), \qquad \ell(x) := \prod_{k=1}^r (x-\hat{x}_k),
\qquad w_k = 1/\ell'(\hat{x}_k).
$$
%
%
Due to the separability of the kernel estimator $\kappa^{(r)}$ in
$(x,y)$ variables, 
the matrix block $A\arrowvert_{\tau \times \sigma}$ can be approximated by a rank-$r$ matrix:
$$
A\arrowvert_{\tau \times \sigma} \approx A^{(r)}\arrowvert_{\tau
  \times \sigma} := L R^{\top}, \qquad L, \, R \in {\mathbb R}^{b
  \times r}, 
$$
where 
$$
L= [L_{ik}]=[\ell(x_i)  \frac{w_k}{x_i -\hat{x}_k}], \qquad
R=[R_{jk}]=[\kappa(\hat{x}_k,y_j)], \quad i \in \tau, \quad j\in \sigma, \quad k \in I_r.
$$
The number of chebyshev points $r$ is hence referred to as the local
rank. 

Pre-computing the barycentric weights $\{ w_1, \dotsc, w_r\}$, we
note that the cost of computing each row of $L$ and $R$
(for a fixed $i \in \tau$ and a fixed $j \in \sigma$ and for all $k=1, \dotsc, r$) is
${\mathcal O}(r)$. It therefore takes ${\mathcal O}(r b)$
floating-point arithmetic operations to form $L$ and $R$. 
Hence, the total cost of a matrix-vector multiplication
$A^{(r)}\arrowvert_{\tau \times \sigma} {\bf w} = L R^{\top} {\bf w}$,
where ${\bf w} \in {\mathbb R}^{b \times 1}$, is ${\mathcal O}(r b)$,
where we first compute $\hat{\bf w} := R^{\top} {\bf w}$ and then $L
\hat{\bf w}$. 

It is to be noted that only admissible blocks are approximated by
low-rank matrices. All inadmissible blocks will be treated as
full-rank $n_{\text{min}} \times n_{\text{min}}$ matrices
whose components are directly given by the kernel $\kappa$.

\medskip
\noindent
{\bf III. Loop over all admissible and inadmissible blocks.} By the
above procedure, we have approximated $A$ by a hierarchical matrix
$A_{\text{H}}$ whose blocks are either of rank at most $r$ or of small
size $n_{\text{min}} \times n_{\text{min}}$. We then assemble all low-rank and full-rank block matrix-vector
multiplications and approximate ${\bf s} = A {\bf z}$ by ${\bf s}_{\text{H}} = A_{\text{H}} {\bf z}$, as follows. 
We first set ${\bf s}_{\text{H}}={\bf 0}$, and loop over all blocks $A\arrowvert_{\tau \times \sigma}$. Let
${\bf s}_{\text{H}}\arrowvert_{\tau}$ and ${\bf z}\arrowvert_{\sigma}$ denote the
parts of vectors ${\bf s}_{\text{H}}$ and ${\bf z}$ corresponding
to the index sets $\tau$ and $\sigma$, respectively. We then proceed
as follows:
\begin{itemize}
\item if $A\arrowvert_{\tau \times \sigma} \in P_{\text{admissible}}$,
  then $A_{\text{H}}\arrowvert_{\tau \times \sigma} =L R^{\top}$ and ${\bf s}_{\text{H}}\arrowvert_{\tau} = {\bf s}_{\text{H}}\arrowvert_{\tau} + L R^{\top}
{\bf z}\arrowvert_{\sigma}$.

\item if $A\arrowvert_{\tau \times \sigma} \in P_{\text{inadmissible}}$,
then ${\bf s}_{\text{H}}\arrowvert_{\tau} = {\bf s}_{\text{H}}\arrowvert_{\tau} + A\arrowvert_{\tau \times \sigma}
{\bf z}\arrowvert_{\sigma}$.
\end{itemize}

\subsubsection{Accuracy and adaptive selection of local ranks}
 \label{sec:accuracy}

We will now discuss the error in the approximation of the kernel matrix $A =
[\kappa(x_i, y_j)] \in
{\mathbb R}^{n_k \times n_k}$ by the hierarchical matrix $A_{\text{H}}$. We will consider the error in Frobenius norm $|| A - A_{\text{H}}
||_{\text{F}}$, where 
$$
|| A ||_{\text{F}}^2 := \sum_{i,j=1}^{n_k} |A_{ij}|^2. 
$$
It is to be noted that other matrix norms, such as
row-sum norm $|| . ||_{\infty}$ associated with the maximum vector norm
and spectral norm $||.||_2$ associated with the Euclidean vector norm,
can similarly be considered. 

We first write the global error as a sum of all local errors \cite{Hackbusch:15}
\begin{equation}\label{global_error}
|| A - A_{\text{H}}||_{\text{F}}^2 = \sum_{\tau \times \sigma \in P_{\text{admissible}}} || A\arrowvert_{\tau \times \sigma}-A_{\text{H}}\arrowvert_{\tau \times \sigma}||_F^2,
\end{equation}
where for each block $\tau \times \sigma \in P_{\text{admissible}}$, the
local error in approximating $A\arrowvert_{\tau \times \sigma} \in
{\mathbb R}^{b \times b}$, with $b = \# \tau = \# \sigma$, reads
\begin{equation}\label{local_error}
|| A\arrowvert_{\tau \times \sigma}-A_{\text{H}}\arrowvert_{\tau
  \times \sigma}||_F^2 = \sum_{i,j = 1}^b | \kappa(x_i,y_j) -
\kappa^{(r)}(x_i,y_j) |^2.
\end{equation}
Each term in the right hand side of \eqref{local_error} is given by the interpolation error (see e.g. \cite{Hackbusch:15}),
$$
\vert \kappa(x,y) - \kappa^{(r)}(x,y) \vert = \frac{| \omega(x) |}{r !} \,
\vert \partial_x^r \kappa(\xi,y) \vert, \qquad  \xi \in  [\min_{i \in
  {\tau}} x_i , \max_{i \in {\tau}} x_i], 
$$
where
$$
\omega(x) = \prod_{k=1}^r (x-\hat{x}_k) \le \Bigl( \frac{\text{diam}
  (\tau)}{4} \Bigr)^r,
$$
and the $x$-derivatives of $\kappa$ satisfies the estimate
\eqref{Asym_Smooth_kernel}, thanks to the asymptotic smoothness of
$\kappa$. 
We further note that for $x \in [\min_{i
  \in {\tau}} x_i , \max_{i \in {\tau}} x_i]$ and $y \in [\min_{j
  \in {\sigma}} y_j , \max_{j \in {\sigma}} y_j]$, we have 
$$
|x-y| \ge \text{dist}(\tau, \sigma).
$$
Hence, assuming $r+s \ge 0$, we get
$$
\vert \kappa - \kappa^{(r)} \vert \le c_0 \, r^{\beta} \, \text{dist}(\tau,
\sigma)^{-s} \, \left( \frac{\alpha}{4} \, \eta \right)^r =: c \, \left( \frac{\alpha}{4} \, \eta \right)^r,
$$
which implies local spectral convergence provided $\eta < 4 /
\alpha$. This can be achieved if we take $\eta_0 < 4/\alpha$ in the admissibility
condition \eqref{admissibility}. For instance, in the numerical
examples in Section \ref{sec:Numerics} we take $\eta_0 =2/\alpha$.

Importantly, as we will show here, the above global and local estimates provide a simple strategy
for a priori selection of local ranks, i.e. the number of Chebyshev
points for each block, to achieve a desired small tolerance $\varepsilon_{\tiny{\text{TOL}}} >
0$ for the global error, 
\begin{equation}\label{global_tol}
|| A - A_{\text{H}}||_{\text{F}} \le \varepsilon_{\tiny{\text{TOL}}}.
\end{equation}
We first observe that \eqref{global_tol} holds if
$$
\vert \kappa - \kappa^{(r)} \vert \le \varepsilon_{\tiny{\text{TOL}}} / n_k.
$$
This simply follows by \eqref{global_error}-\eqref{local_error} and
noting that $\sum_{\tau \times \sigma \in P_{\text{admissible}}} b^2 <
n_{k}^2$. Hence, we can achieve \eqref{global_tol} if we set 
$$
c \, \left( \frac{\alpha}{4} \, \eta \right)^r \le
\varepsilon_{\tiny{\text{TOL}}} / n_k.
$$
We hence obtain an optimal local rank
\begin{equation}\label{optimal_local_rank}
r = \Biggl\lceil   \frac{\log(\varepsilon_{\tiny{\text{TOL}}}/(c \, n_k)) }{\log(\alpha
  \, \eta / 4)} \Biggr\rceil.
\end{equation}
We note that each block may have a different $\eta$ and hence a
different local rank $r$. We also note that the optimal local rank
depends logarithmically on $n_k$.

\subsubsection{Computational cost}
\label{sec:total-cost}

We now discuss the overall complexity of the proposed strategy (steps 1
and 2) applied to each iteration of Sinkhorn's algorithm.

As discussed earlier in Section \ref{sec:part1}, the complexity of the original $d$-dimensional problems $Q {\bf
  w}$ and $\hat{Q} {\bf w}$ reduces to several smaller one-dimensional
problems of either $Q^{(k)} {\bf z}$ or $\hat{Q}^{(k)} {\bf
  z}$ form, given by \eqref{cost_step1}. 
We now study the complexity of computing one-dimensional
problems to obtain $\text{cost}(Q^{(k)} {\bf z})$. 
We follow \cite{Hackbusch:15} and assume $n_k = 2^q$, where $q$ is a
non-negative integer that may be different for different $k$. 
Let ${\mathcal H}_q$ denote the family of $2^q \times 2^q$
hierarchical matrices constructed above. 
Corresponding to the quadtree block
structure, the representation of ${\mathcal H}_q$ for $q \ge 1$ can be given recursively,
\begin{equation}\label{recursive_forms} 
{\mathcal H}_q = 
\left(\begin{array}{@{}c|c@{}}
  {\mathcal H}_{q-1}
  & {\mathcal N}_{q-1}\\
\hline
  {\mathcal N}_{q-1}^*
& {\mathcal H}_{q-1}
\end{array}\right),
\qquad 
{\mathcal N}_q = 
\left(\begin{array}{@{}c|c@{}}
  {\mathcal R}_{q-1}
  & {\mathcal R}_{q-1}\\
\hline
  {\mathcal N}_{q-1}
& {\mathcal R}_{q-1}
\end{array}\right),
\end{equation}
where ${\mathcal R}_{q}$ is the family of rank-$r$ square matrices of size $2^q
\times 2^q$. Note that for $q=0$, the matrices ${\mathcal H}_q$,
${\mathcal N}_q$, ${\mathcal R}_q$ are $1 \times 1$
matrices and hence are simply scalars.

Now let $W_{{\mathcal H}}(q)$ be the cost of multiplying ${\mathcal H}_q$ by a vector
${\bf z}$. The recursive forms \eqref{recursive_forms} imply
$$
W_{{\mathcal H}}(q) = 2 \, W_{{\mathcal H}}(q-1) + 2 \, W_{{\mathcal N}}(q-1), \qquad
W_{{\mathcal N}}(q) = W_{{\mathcal N}}(q-1) + 3 \, W_{{\mathcal R}}(q-1),
$$
where $W_{{\mathcal N}}(q)$ and $W_{{\mathcal R}}(q)$ are the costs of
multiplying a vector by ${\mathcal N}_q$ and ${\mathcal R}_q$,
respectively. 
We first note that since ${\mathcal R}_q$ is a rank-$r$ matrix, we have  
$$
 W_{{\mathcal R}}(q) = c \, r \, 2^q, \qquad c >0.
$$ 
We can then solve the recursive equation $W_{{\mathcal N}}(q) =
W_{{\mathcal N}}(q-1) + 3 \, c \, r \, 2^{q-1}$ with $W_{{\mathcal N}}(0)
= 1$ to get
$$
W_{{\mathcal N}}(q) = 3 \, c \, r \, (2^q - 1) + 1 \sim 3 \, c \, r \, 2^q.
$$
We finally solve the recursive equation $W_{{\mathcal H}}(q) = 2 \,
W_{{\mathcal H}}(q-1) + 6 \, c \, r \, 2^{q-1}$ with $W_{{\mathcal H}}(0)
= 1$ to obtain
$$
W_{{\mathcal H}}(q) = (1 + 3 \, c \, r \, q ) \, 2^q.
$$
If we choose the local ranks $r$ according to
\eqref{optimal_local_rank}, then $r = {\mathcal O}(\log
n_k)$, and hence the cost of computing the one-dimensional problems
$Q^{(k)}{\bf z}$ is
\begin{equation}\label{cost_step2}
\text{cost}(Q^{(k)}{\bf z})= {\mathcal
  O}(n_k \log^2 n_k).
\end{equation}
%
%
Combining \eqref{cost_step1} and \eqref{cost_step2}, we obtain
$$
\text{cost}(Q {\bf w})= \sum_{k=1}^d \tilde{n}_k \, {\mathcal
  O}(n_k \log^2 n_k) = {\mathcal
  O}(n \sum_{k=1}^d \log^2 n_k ) \le {\mathcal
  O}(n (\sum_{k=1}^d \log n_k)^2 ) ={\mathcal
  O}(n \log^2 n ).
$$

We note that with such a reduction in the complexity of each Sinkhorn iteration, 
the overall cost of computing Sinkhorn divergence after $K = {\mathcal
  O}(\log n)$ iterations will reduce to ${\mathcal O}(n \, \log^3 n)$.

\subsection{Algorithm}
\label{sec:algorithm}

Algorithm \ref{ALG} outlines the hierarchical low-rank computation of
Sinkhorn divergence. 
\begin{algorithm}[!ht]
\caption{ Hierarchical low-rank computation of Sinkhorn divergence}
\label{ALG}
\begin{algorithmic} 
\medskip
\STATE {\bf Input:} 
\hskip 0.2cm $\bullet$ cost matrix $C \in {\mathbb
R}_+^{n \times n}$ generated
by metric functions $c_k:{\mathbb R} \times {\mathbb R} \rightarrow {\mathbb
  R}_+$, $k=1, \dotsc, d$

\hskip 1.67cm $\bullet$ probability vectors ${\bf f}, {\bf g}
\in \Sigma_n$ supported on two bounded subsets of ${\mathbb R}^d$

\hskip 1.67cm $\bullet$ tolerances $\varepsilon_{\text{TOL}}$ and $\varepsilon_{\text{S}}$
 
\medskip
\STATE {\bf Output:} \hskip 0.2cm Sinkhorn divergence $S_{p, \lambda}
({\bf f}, {\bf g})$

\medskip
\medskip
\STATE {\bf 1.} Select $\lambda>0$; see Remark \ref{remark:lambda}.

\medskip
\STATE {\bf 2.} For every $k=1, \dotsc, d$, construct two hierarchical
matrices (within tolerance $\varepsilon_{\text{TOL}}$):
$$
Q_{\text{H}}^{(k)} \approx Q^{(k)} =[\exp(- \lambda \, c_k)], \qquad 
\hat{Q}_{\text{H}}^{(k)} \approx \hat{Q}^{(k)} = [c_k \, \exp(- \lambda \, c_k)].
$$

\STATE {\bf 3.} Implement two procedures, following steps 1-2 in Section \ref{sec:H-Sinkhorn}:

\vskip 0.2cm
\hskip 1cm Procedure 1. \ \  ${\bf s}$=MVM1($Q_{\text{H}}^{(1)}, \dotsc, Q_{\text{H}}^{(d)},
{\bf z}$) \ \ \ producing the result ${\bf s} \approx Q \, {\bf z}$

\vskip 0.2cm
\hskip 1cm Procedure 2. \ \  ${\bf s}$=MVM2($Q_{\text{H}}^{(1)}, \dotsc,
Q_{\text{H}}^{(d)},\hat{Q}_{\text{H}}^{(1)}, \dotsc,
\hat{Q}_{\text{H}}^{(d)}, {\bf z}$) \ \ \ producing the result ${\bf s} \approx \hat{Q} \, {\bf z}$

\vskip 0.1cm
\medskip 
\STATE {\bf 4.} Sinkhorn's loop:

\vskip 0.1cm
\hskip 1cm ${\bf u}^{(0)} = \mathds{1}_n /n$

\vskip 0.1cm
\hskip 1cm ${\bf v}^{(0)} =  {\bf g} \, \oslash \, \text{MVM1}(Q_{\text{H}}^{(1)}, \dotsc, Q_{\text{H}}^{(d)},
{\bf u}^{(0)})$

\vskip 0.1cm
\hskip 1cm $i=0$

\vskip 0.15cm
\hskip 1cm {\bf while} \ $\langle$ stopping criterion
\eqref{stop_criterion} does not hold
$\rangle$ \ {\bf do}

\vskip 0.1cm
\hskip 1.8cm ${\bf u}^{(i+1)} =  {\bf f} \, \oslash \, \text{MVM1}(Q_{\text{H}}^{(1)}, \dotsc, Q_{\text{H}}^{(d)},
{\bf v}^{(i)})$

\vskip 0.1cm
\hskip 1.8cm ${\bf v}^{(i+1)} =  {\bf g} \, \oslash \, \text{MVM1}(Q_{\text{H}}^{(1)}, \dotsc, Q_{\text{H}}^{(d)},
{\bf u}^{(i)})$

\vskip 0.1cm
\hskip 1.8cm $i = i+1$

\vskip 0.1cm
\hskip 1cm {\bf end while} 

\vskip 0.15cm
\hskip 1cm ${\bf s} = \text{MVM2}(Q_{\text{H}}^{(1)}, \dotsc,
Q_{\text{H}}^{(d)},\hat{Q}_{\text{H}}^{(1)}, \dotsc, \hat{Q}_{\text{H}}^{(d)}, {\bf v}^{(i)})$ 

\vskip 0.1cm
\hskip 1cm $S_{\lambda,p} = \left(   {\bf s}^{\top}  {\bf u}^{(i)} \right)^{1/p}$

\medskip
\end{algorithmic}
\end{algorithm}

\section{Numerical illustration}
\label{sec:Numerics}

In this section we will present a numerical example to verify the theoretical and
numerical results presented in previous sections. In particular, we will
consider an example studied in
\cite{Motamed_Appelo:19} and demonstrate the applicability of Sinkhorn
divergence and the proposed algorithm in solving optimization
problems. 

\medskip
\noindent
{\bf Problem statement.} 
Let ${\bf f} = [f(x_i)] \in {\mathbb R}^n$ be a three-pulse signal given at $n$ discrete
points $x_i = (i-1)/(n-1)$, with $i=1, \dotsc, n$, where
\[
f(x) =  e^{-\left( \frac{x-0.4}{\sigma}\right)^2} 
- e^{-\left( \frac{x-0.5}{\sigma}\right)^2} 
+e^{-\left( \frac{x-0.6}{\sigma}\right)^2}, \qquad x\in[0,1].
\] 
Here, $\sigma >0$ is a given positive constant that controls the
spread of the three pulses; the smaller $\sigma$, the narrower the three
pulses. 
Let further ${\bf g}(s)=[g(x_i; s)] \in {\mathbb R}^n$ be a shifted version of ${\bf
  f}$ with a shift $s$ given at the same discrete points $x_i =  (i-1)/(n-1) $, with $i=1, \dotsc, n$, where
\[
g(x; s) =  e^{-\left( \frac{x-s-0.4}{\sigma}\right)^2} 
- e^{-\left( \frac{x-s-0.5}{\sigma}\right)^2} 
+e^{-\left( \frac{x-s-0.6}{\sigma}\right)^2}, \qquad x\in[0,1].
\] 
Assuming $-0.3 \le s \le 0.3$, our goal is to find the ``optimal'' shift $s^*$ that minimizes the ``distance''
between the original signal ${\bf f}$ and the shifted signal ${\bf
  g}(s)$,
$$
s^* = \argmin_{s\in[-0.3,0.3]} \text{d}({\bf f}, {\bf g}(s)).
$$
Here, $\text{d}(.,.)$ is a ``loss'' function that measures dissimilarities between ${\bf f}$ and ${\bf g}(s)$; see the
discussion below. While being a simple problem with a known
solution $s^* = 0$, this model problem serves as an
illustrative example exhibiting important challenges that we may face in
solving optimization problems. 
Figure \ref{fig:signals} shows the original signal ${\bf f}$ and two shifted signals ${\bf
    g}(s)$ for two different shifts $s=-0.3$ and $s=0.3$. Wide signals
  (left) correspond to width $\sigma=0.05$, and narrow signals (right) correspond to width $\sigma=0.01$.
\begin{figure}[!ht]
\begin{center}
\includegraphics[width=0.45\textwidth]{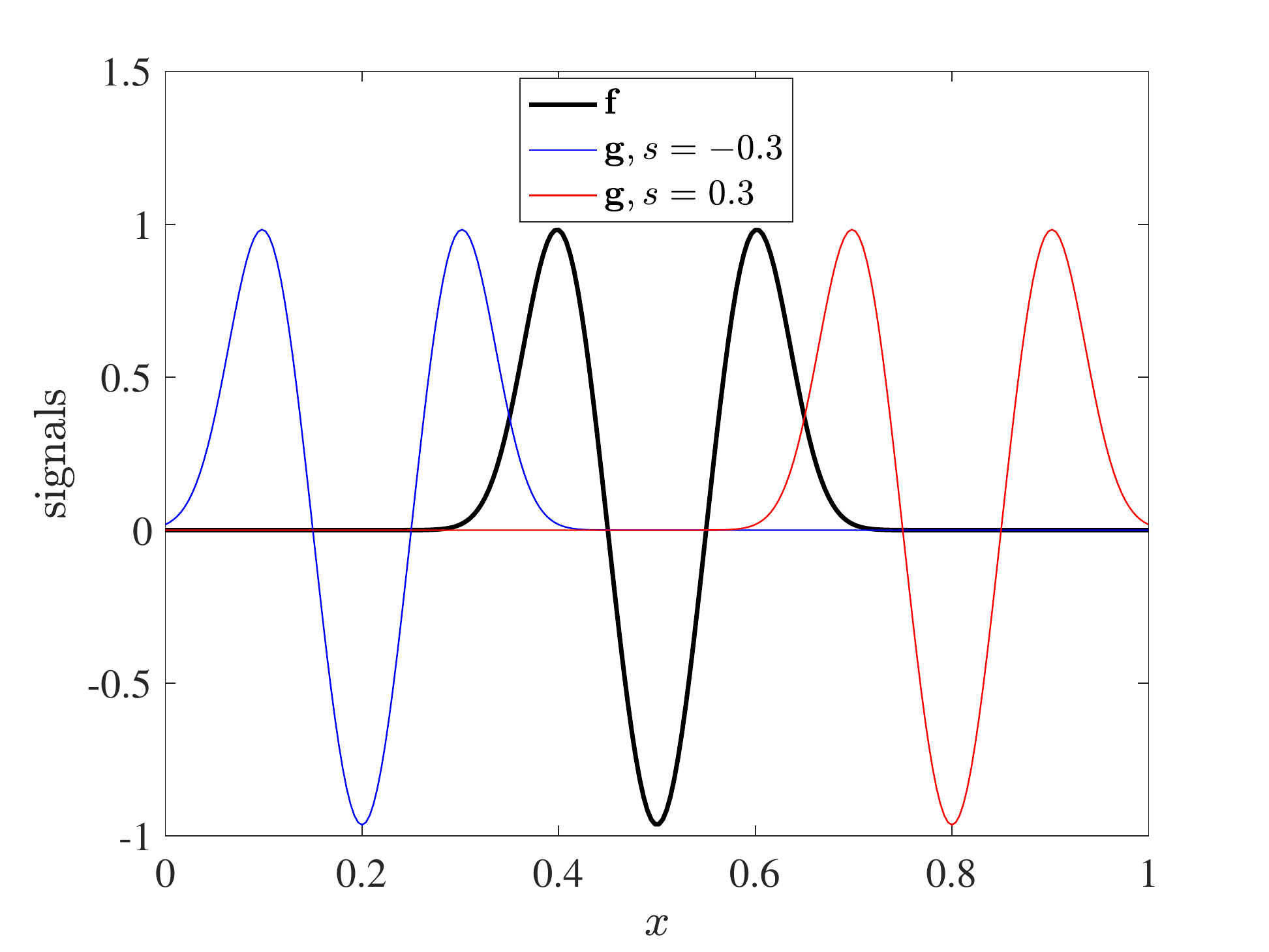}      
\includegraphics[width=0.45\textwidth]{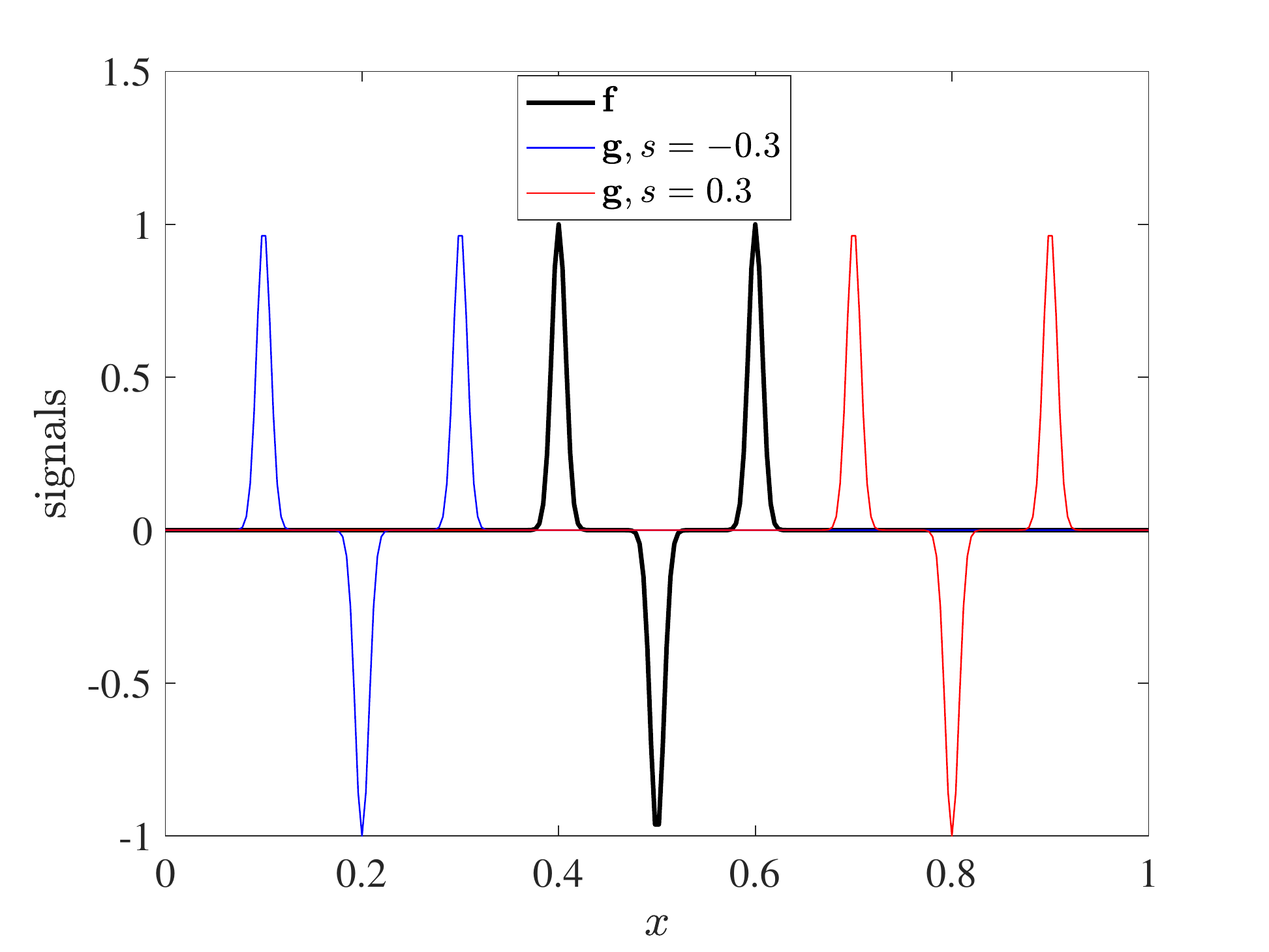}     
\caption{The original signal ${\bf f}$ and the shifted signal ${\bf
    g}(s)$ for two different shifts $s=-0.3$ and $s=0.3$. 
Left figure shows wide signals ($\sigma = 0.05$), and right figure shows
narrow signals ($\sigma = 0.01$).}  
\label{fig:signals}
\end{center}
\end{figure}

\medskip
\noindent
{\bf Loss functions.} 
A first step is tackling an optimization problem is to select an appropriate
loss function. Here, we consider and compare three loss functions:
\begin{itemize}
\item Euclidean distance (or $L^2$ norm): $\text{d}_{\text{E}}({\bf f}, {\bf g}(s)) := ||  {\bf f} - {\bf g}(s)
  ||_2$.

\item Quadratic Wasserstein distance: $\text{d}_{\text{W}}({\bf
    f}, {\bf g}(s))$ defined in \eqref{d_W}.

\item Quadratic Sinkhorn divergence: $\text{d}_{\text{S}}({\bf f}, {\bf
    g}(s))$ defined in \eqref{d_S}.
\end{itemize}
It is to be noted that since the two signals are not necessarily
probability vectors, i.e. positivity and mass balance are not
expected, the signals need to be pre-processed before we can measure
their Wasserstein distance and Sinkhorn divergence. While there are
various techniques for this, here we employ a sign-splitting and
normalization technique that considers the non-negative parts ${\bf f}^+$ and
${\bf g}^+$ and the non-positive parts ${\bf f}^-$ and ${\bf g}^-$ of
the two signals separately. Precisely, we define
\begin{equation}\label{d_W}
d_{\text{W}}({\bf f}, {\bf g}(s)) := W_2 \left( \frac{{\bf f}^+}{{\bf f}^{+\top}
    \mathds{1}}, \frac{{\bf g}(s)^+}{{\bf g}(s)^{+\top}
    \mathds{1}} \right) + W_2 \left( \frac{{\bf f}^-}{{\bf f}^{-\top}
    \mathds{1}}, \frac{{\bf g}(s)^-}{{\bf g}(s)^{-\top}
    \mathds{1}} \right),
\end{equation}
\begin{equation}\label{d_S}
d_{\text{S}}({\bf f}, {\bf g}(s)) := S_{2,  \lambda} \left( \frac{{\bf f}^+}{{\bf f}^{+\top}
    \mathds{1}}, \frac{{\bf g}(s)^+}{{\bf g}(s)^{+\top}
    \mathds{1}} \right) + S_{2,  \lambda} \left( \frac{{\bf f}^-}{{\bf f}^{-\top}
    \mathds{1}}, \frac{{\bf g}(s)^-}{{\bf g}(s)^{-\top}
    \mathds{1}} \right). 
\end{equation} 
Here, $W_2$ and $S_{2,  \lambda} $ are respectively given by \eqref{W}
and \eqref{S} with $p=2$.

Two desirable properties of a loss function include its convexity with
respect to the unknown parameter (here $s$) and its efficient computability
for many different values of $s$. 
While the Euclidean distance $d_{\text{E}}$ is widely used and can be
efficiently computed with a cost ${\mathcal O}(n)$, it may generate multiple local extrema that could result in
converging to wrong solutions. The Wasserstein
metric, on the other hand, features better convexity and hence better convergence to
correct solutions; see \cite{Motamed_Appelo:19} and references there
in. 
However, the application of Wasserstein metric is
limited to low-dimensional signals due to its high computational cost and infeasibility in high
dimensions. This makes the Wasserstein loss function $\text{d}_{\text{W}}$ particularly
suitable for one-dimensional problems, where it can be efficiently computed with a cost
${\mathcal O}(n \, \log n)$, as discussed in Section \ref{sec:Wass}. 
As a promising alternative to
$\text{d}_{\text{E}}$ and $\text{d}_{\text{W}}$, we may consider the
Sinkhorn loss function $\text{d}_{\text{S}}$ that possesses both desirable properties of a loss function. In terms of convexity,
Sinkhorn divergence is similar to the Wasserstein metric. In terms
of computational cost, as shown in the present work, we can
efficiently compute Sinkhorn divergence in multiple dimensions with a
near-linear cost ${\mathcal O}(n \, \log^3 n)$.

\medskip
\noindent
{\bf Numerical results and discussion.} 
Figure \ref{fig:distances_comp} shows the loss
functions ($d_{\text{E}}, d_{\text{W}}, d_{\text{S}}$) versus $s$. Here, we set $n=2^{12}$ and consider two different values for the width of signals,
$\sigma = 0.05$ (left figure) and $\sigma = 0.01$ (right figure). 
We compute the Sinkhorn loss function with $\lambda = 50$ both by the
original Sinkhorn's algorithm (labeled $d_{\text{S}}$ in the figure) and
by the proposed fast algorithm outlined in Algorithm \ref{ALG} (labeled $d_{\text{S}}^{\text{H}}$ in
the figure). In computing $d_{\text{S}}^{\text{H}}$, we obtain the
optimal local ranks by \eqref{optimal_local_rank}, with parameters
$c=1$ and $\alpha = 2$ that correspond to the $L^p$ cost function with
$p=2$; see the
discussion on $L^p$ cost functions in Section \ref{sec:assumptions}
for this choice of parameters. 
We consider the accuracy tolerances
$\varepsilon_{\tiny{\text{TOL}}} = \varepsilon_{\tiny{\text{S}}}  = 0.01$. 
\begin{figure}[!ht]
\begin{center}
\includegraphics[width=0.42\textwidth]{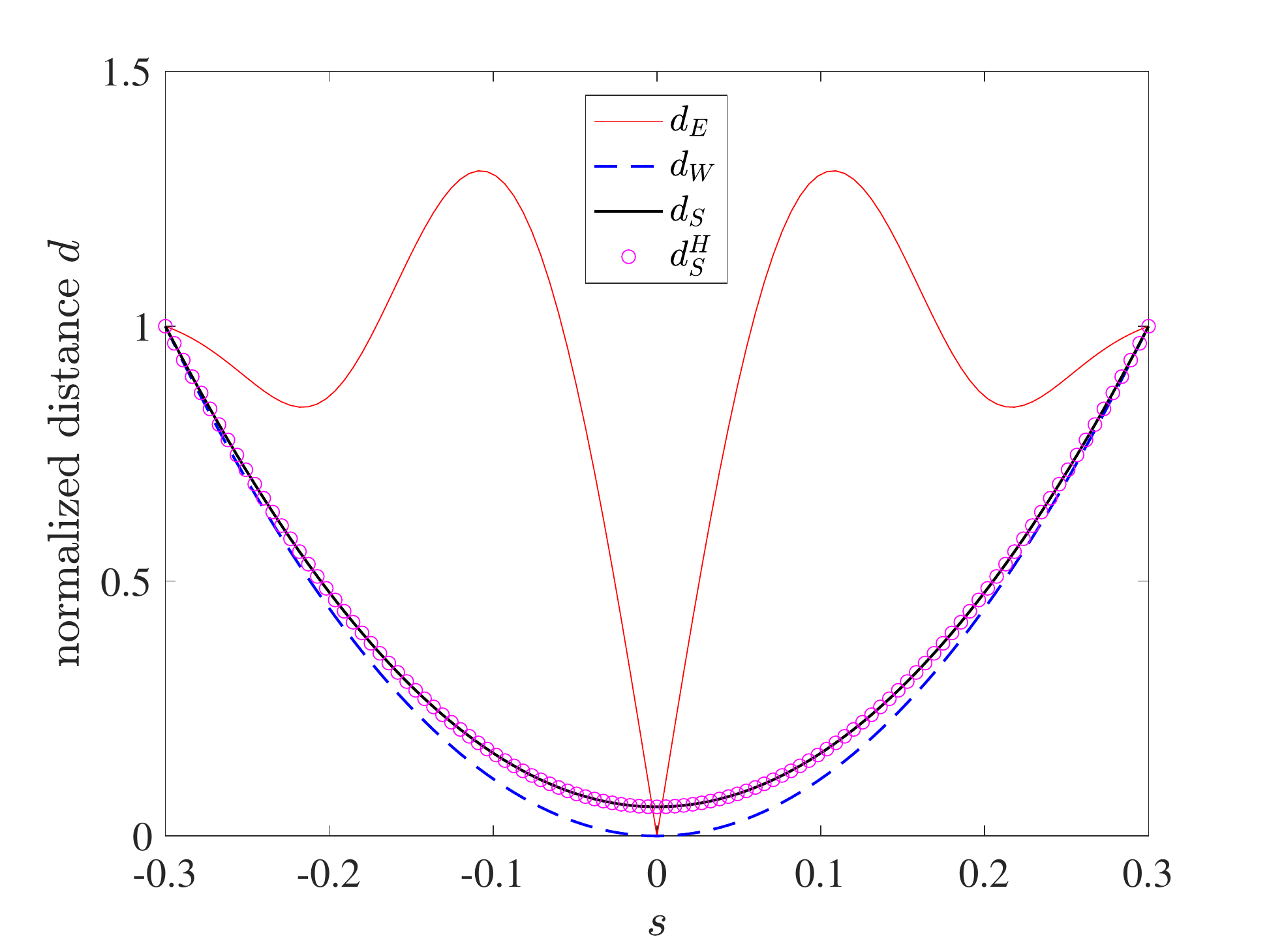}      
\includegraphics[width=0.42\textwidth]{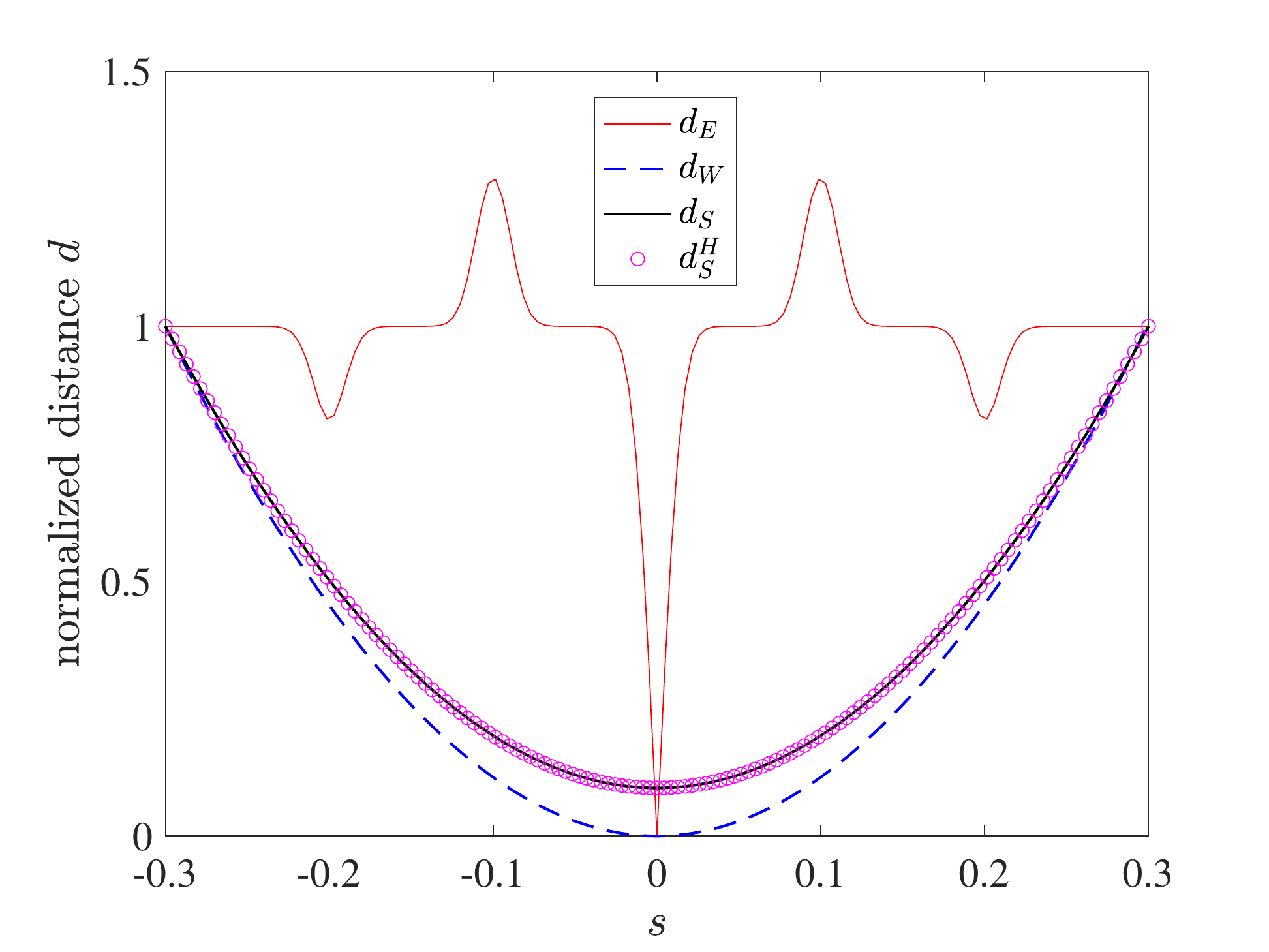}     
\caption{Comparison of different loss functions that measure the
  dissimilarity between ${\bf f}$ and ${\bf g}(s)$ as a function of the shift
$s$. While the Euclidean loss function has local extrema, the
Wasserstein and Sinkhorn loss functions are convex. 
Left figure is for wide signals ($\sigma = 0.05$), and right figure is for
narrow signals ($\sigma = 0.01$).}  
\label{fig:distances_comp}
\end{center}
\end{figure}

As we observe, the Euclidean loss function has local extrema, while the
Wasserstein and Sinkhorn loss functions are convex. 
We also observe that $d_{\text{S}}^{\text{H}}$ well approximates
$d_{\text{S}}$, verifying the accuracy of the proposed algorithm in computing Sinkhorn
divergence.

Figure \ref{fig:cost} shows the CPU time of computing
$d_{\text{S}}^{\text{H}}$ as a function of $n$. We observe that the
cost is ${\mathcal O}(n \, \log^3 n)$ as expected.
\begin{figure}[!h]
\begin{center}
\includegraphics[width=0.5\linewidth]{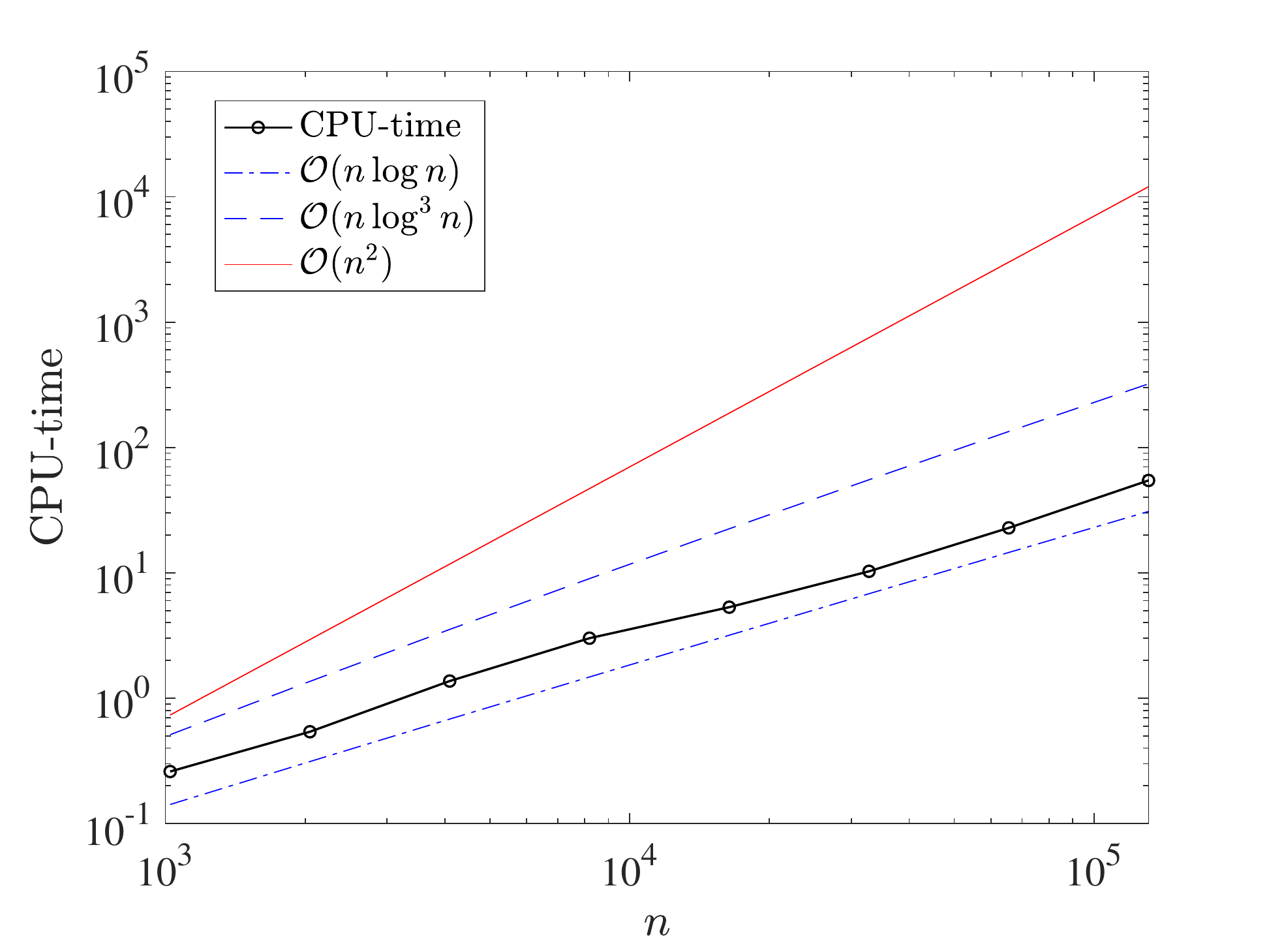}     
\caption{CPU time versus $n$.}
\label{fig:cost}
\end{center}
\end{figure}
%

\section{Conclusion}
\label{sec:CON}

This work presents a fast method for computing Sinkhorn
divergence between two discrete probability vectors supported on
possibly high-dimensional spaces. The cost matrix in optimal transport
problem is assumed to be decomposable into a sum
of asymptotically smooth Kronecker product factors. 
The method combines Sinkhorn's matrix
scaling iteration with a low-rank hierarchical representation of the
scaling matrices to achieve a near-linear complexity. 
This provides a fast and easy-to-implement algorithm for computing
Sinkhorn divergence, enabling the applicability of Sinkhorn divergence to large-scale optimization problems, where the computation of classical
Wasserstein metric is not feasible. 

Future directions include 
%
exploring the application of the proposed
hierarchical low-rank Sinkhorn's algorithm to large-scale Bayesian
inversion and other data science and machine learning problems.


\bigskip
\noindent
{\bf Acknowledgments.} I would like to thank Klas Modin for a question
he raised during a seminar that I gave at Chalmers University of Technology
in Gothenburg, Sweden. Klas' question on the computation of Wasserstein metric in multiple
dimensions triggered my curiosity and 
persuaded me to work further on the
subject. 

\section*{Appendix}
In this appendix we collect a number of auxiliary lemmas.

\begin{lemma}\label{lem1}
For a real matrix $A=[A_{ij}]$, let $\exp[A]$ denote the matrix obtained by the entrywise
application of exponential function to $A$, i.e. $\exp[A] :=
[\exp(A_{ij})]$. Then we have
$$
\exp[C^{(d)} \oplus \dotsi \oplus C^{(1)}] = \exp[C^{(d)}] \otimes
\dotsi \otimes \exp[C^{(1)}].
$$
\end{lemma}
\begin{proof}
The case $d=1$ is trivial. Consider the case $d=2$. By the definition
of Kronecker sum, we have
$$
\exp[C^{(2)} \oplus C^{(1)}] = \exp[C^{(2)} \otimes J_1 + J_2 \otimes
C^{(1)}] =
$$

$$
=
\left(\begin{array}{@{}c|c|c@{}}
  \exp(C_{11}^{(2)}) J_1
  & \cdots & \exp(C_{1n_2}^{(2)}) J_1\\
\hline
  \vdots &
  \ddots & \vdots\\
\hline
  \exp(C_{n_2 1}^{(2)}) J_1
& \cdots & \exp(C_{n_2 n_2}^{(2)}) J_1
\end{array}\right)
\odot
\left(\begin{array}{@{}c|c|c@{}}
  \exp[C^{(1)}]
  & \cdots & \exp[C^{(1)}]\\
\hline
  \vdots &
  \ddots & \vdots\\
\hline
  \exp[C^{(1)}]
& \cdots & \exp[C^{(1)}]
\end{array}\right)
=
$$

$$
= 
\left(\begin{array}{@{}c|c|c@{}}
  \exp(C_{11}^{(2)}) \exp[C^{(1)}]
  & \cdots & \exp(C_{1n_2}^{(2)}) \exp[C^{(1)}]\\
\hline
  \vdots &
  \ddots & \vdots\\
\hline
  \exp(C_{n_2 1}^{(2)}) \exp[C^{(1)}]
& \cdots & \exp(C_{n_2 n_2}^{(2)}) \exp[C^{(1)}]
\end{array}\right)
=
$$

$$
= \exp[C^{(2)}] \otimes \exp[C^{(1)}].
$$
Now assume that it is true for $d-1$: 
$$
\exp[C^{(d-1)} \oplus \dotsi \oplus C^{(1)}] = \exp[C^{(d-1)}] \otimes
\dotsi \otimes \exp[C^{(1)}].
$$
Then
$$
\exp[C^{(d)} \oplus \dotsi \oplus C^{(1)}] = \exp[C^{(d)}] \otimes
\exp[C^{(d-1)} \oplus \dotsi \oplus C^{(1)}] =
$$

$$
= \exp[C^{(d)}] \otimes
\dotsi \otimes \exp[C^{(1)}].
$$
\end{proof}

\begin{lemma}\label{lem2}
We have
$$
[C^{(d)} \oplus \dotsi \oplus C^{(1)}] \odot [Q^{(d)} \otimes
\dotsi \otimes Q^{(1)}] =
\sum_{k=1}^d A_k^{(d)} \otimes
\dotsi \otimes A_k^{(1)}, 
\quad 
A_k^{(m)} =
\left\{ \begin{array}{ll}
C^{(m)} \odot Q^{(m)}, &  k = m \\
Q^{(m)}, & k \neq m 
\end{array} \right. 
$$
\end{lemma}
\begin{proof}
The case $d=1$ is trivial. Consider the case $d=2$. By the definition
of Kronecker sum, we have
$$
[C^{(2)} \oplus C^{(1)}] \odot [Q^{(2)} \otimes Q^{(1)}] =
[C^{(2)} \oplus J_1] \odot [Q^{(2)} \otimes Q^{(1)}] +
[J_2 \oplus C^{(1)}] \odot [Q^{(2)} \otimes Q^{(1)}] 
$$

$$
=
\left(\begin{array}{@{}c|c|c@{}}
  C_{11}^{(2)} J_1
  & \cdots & C_{1n_2}^{(2)} J_1\\
\hline
  \vdots &
  \ddots & \vdots\\
\hline
  C_{n_2 1}^{(2)} J_1
& \cdots & C_{n_2 n_2}^{(2)} J_1
\end{array}\right)
\odot
\left(\begin{array}{@{}c|c|c@{}}
  Q_{11}^{(2)} Q^{(1)}
  & \cdots & Q_{1n_2}^{(2)} Q^{(1)}\\
\hline
  \vdots &
  \ddots & \vdots\\
\hline
  Q_{n_2 1}^{(2)} Q^{(1)}
& \cdots & Q_{n_2 n_2}^{(2)} Q^{(1)}
\end{array}\right)
+
$$

$$
\left(\begin{array}{@{}c|c|c@{}}
  C^{(1)}
  & \cdots & C^{(1)}\\
\hline
  \vdots &
  \ddots & \vdots\\
\hline
 C^{(1)}
& \cdots & C^{(1)}
\end{array}\right)
\odot
\left(\begin{array}{@{}c|c|c@{}}
  Q_{11}^{(2)} Q^{(1)}
  & \cdots & Q_{1n_2}^{(2)} Q^{(1)}\\
\hline
  \vdots &
  \ddots & \vdots\\
\hline
  Q_{n_2 1}^{(2)} Q^{(1)}
& \cdots & Q_{n_2 n_2}^{(2)} Q^{(1)}
\end{array}\right)
=
$$

$$
=
\left(\begin{array}{@{}c|c|c@{}}
  C_{11}^{(2)} Q_{11}^{(2)} Q^{(1)}
  & \cdots & C_{1n_2}^{(2)} Q_{1n_2}^{(2)} Q^{(1)}\\
\hline
  \vdots &
  \ddots & \vdots\\
\hline
  C_{n_2 1}^{(2)} Q_{n_2 1}^{(2)} Q^{(1)}
& \cdots & C_{n_2 n_2}^{(2)} Q_{n_2 n_2}^{(2)} Q^{(1)}
\end{array}\right)
+
\left(\begin{array}{@{}c|c|c@{}}
  Q_{11}^{(2)} (C^{(1)} \odot Q^{(1)})
  & \cdots & Q_{1n_2}^{(2)} (C^{(1)} \odot Q^{(1)})\\
\hline
  \vdots &
  \ddots & \vdots\\
\hline
  Q_{n_2 1}^{(2)} (C^{(1)} \odot Q^{(1)})
& \cdots & Q_{n_2 n_2}^{(2)} (C^{(1)} \odot Q^{(1)})
\end{array}\right)
=
$$

$$
=[C^{(2)} \odot Q^{(2)}] \otimes Q^{(1)} + Q^{(2)} \otimes [C^{(1)} \odot Q^{(1)}].
$$
Now assume that it is true for $d-1$: 
$$
[C^{(d-1)} \oplus \dotsi \oplus C^{(1)}] \odot [Q^{(d-1)} \otimes
\dotsi \otimes Q^{(1)}] =
\sum_{k=1}^{d-1} A_k^{(d-1)} \otimes
\dotsi \otimes A_k^{(1)}.
$$
Then
$$
[C^{(d)} \oplus \dotsi \oplus C^{(1)}] \odot [Q^{(d)} \otimes
\dotsi \otimes Q^{(1)}] =
$$

$$
[C^{(d)} \odot Q^{(d)}] \otimes [Q^{(d-1)} \otimes \dotsi \otimes Q^{(1)}]
+ Q^{(d)} \otimes 
\left( [C^{(d-1)} \oplus \dotsi \oplus C^{(1)}] \odot [Q^{(d-1)} \otimes
\dotsi \otimes Q^{(1)}] \right) =
$$

$$
[C^{(d)} \odot Q^{(d)}] \otimes Q^{(d-1)} \otimes \dotsi \otimes
Q^{(1)} 
+ 
Q^{(d)} \otimes  \sum_{k=1}^{d-1} A_k^{(d-1)} \otimes
\dotsi \otimes A_k^{(1)} =
$$

$$
A_d^{(d)} \otimes Q^{(d-1)} \otimes \dotsi \otimes
Q^{(1)} 
+ 
\sum_{k=1}^{d-1} Q^{(d)} \otimes A_k^{(d-1)} \otimes
\dotsi \otimes A_k^{(1)} =
$$

$$
=\sum_{k=1}^d A_k^{(d)} \otimes
\dotsi \otimes A_k^{(1)},
$$
noting that $A_d^{(d)} = C^{(d)} \odot Q^{(d)}$, and $A_d^{(k)} =
Q^{(k)}$ and $A_k^{(d)} = Q^{(d)}$ for $k=1, \dotsc, d-1$.

\end{proof}

\bibliographystyle{plain}
\bibliography{refs}

\begin{thebibliography}{10}

\bibitem{Altschuler_etal:17}
J.~Altschuler, J.~Weed, and P.~Rigollet.
\newblock Near-linear time approximation algorithms for optimal transport via
  {S}inkhorn iteration.
\newblock In I.~Guyon, U.~V. Luxburg, S.~Bengio, H.~Wallach, R.~Fergus,
  S.~Vishwanathan, and R.~Garnett, editors, {\em Advances in Neural Information
  Processing Systems 30}, pages 1961--1971, 2017.

\bibitem{Brenier:91}
Y.~Brenier.
\newblock Polar factorization and monotone rearrangement of vector-valued
  functions.
\newblock {\em Comm. Pure Appl. Math.}, 44:375--417, 1991.

\bibitem{Cuturi:13}
M.~Cuturi.
\newblock Sinkhorn distances: Lightspeed computation of optimal transport.
\newblock In {\em Advances in Neural Information Processing Systems 26}, pages
  2292--2300, 2013.

\bibitem{Dvurechensky_etal:18}
P.~Dvurechensky, A.~Gasnikov, and A.~Kroshnin.
\newblock Computational optimal transport: Complexity by accelerated gradient
  descent is better than by {S}inkhorn's algorithm.
\newblock In J.~Dy and A.~Krause, editors, {\em Proceedings of the 35th
  International Conference on Machine Learning, PMLR 80}, pages 1367--1376,
  2018.

\bibitem{engquist2016optimal}
B.~Engquist, B.D.~Froese Brittany, and Y.~Yang.
\newblock Optimal transport for seismic full waveform inversion.
\newblock {\em Communications in Mathematical Sciences}, 14(8):2309--2330,
  2016.

\bibitem{FranklinLorenz:89}
J.~Franklin and J.~Lorenz.
\newblock On the scaling of multidimensional matrices.
\newblock {\em Linear Algebra and its Application}, 114:717--735, 1989.

\bibitem{Hackbusch:15}
W.~Hackbusch.
\newblock {\em Hierarchical Matrices: Algorithms and Analysis}, volume~49 of
  {\em Springer Series in Computational Mathematics}.
\newblock Springer-Verlag, 2015.

\bibitem{Indritz:61}
J.~Indritz.
\newblock An inequality for {H}ermite polynomials.
\newblock {\em Proc. Amer. Math. Soc.}, 12:981--983, 1961.

\bibitem{Kalantari-Khachiyan:96}
B.~Kalantari and L.~Khachiyan.
\newblock On the complexity of nonnegative-matrix scaling.
\newblock {\em Linear Algebra and its Applications}, 240:87--103, 1996.

\bibitem{Knight:08}
P.~A Knight.
\newblock The {S}inkhorn–{K}nopp algorithm: convergence and applications.
\newblock {\em SIAM J. on Matrix Analysis and Applications}, 30:261--275, 2008.

\bibitem{Motamed_Appelo:19}
M.~Motamed and D.~Appel{\"o}.
\newblock Wasserstein metric-driven {B}ayesian inversion with applications to
  signal processing.
\newblock {\em International J. for Uncertainty Quantification}, 9:395--414,
  2019.

\bibitem{Nemirovski-Rothblum:99}
A.~Nemirovski and U.~Rothblum.
\newblock On complexity of matrix scaling.
\newblock {\em Linear Algebra and its Applications}, 302:435--460, 1999.

\bibitem{Peyre_Cuturi:19}
G.~Peyr{\'e} and M.~Cuturi.
\newblock Computational optimal transport.
\newblock {\em Foundations and Trends in Machine Learning}, 11:355--607, 2019.

\bibitem{Sinkhorn:64}
R.~Sinkhorn.
\newblock A relationship between arbitrary positive matrices and doubly
  stochastic matrices.
\newblock {\em Annals of Mathematical Statististics}, 35:876--879, 1964.

\bibitem{Solomon_etal:15}
J.~Solomon, F.~De Goes, G.~Peyr{\'e}, M.~Cuturi, A.~Butscher, A.~Nguyen, T.~Du,
  and L.~Guibas.
\newblock Convolutional {W}asserstein distances: efficient optimal
  transportation on geometric domains.
\newblock {\em ACM Transactions on Graphics}, 34:66:1--66:11, 2015.

\bibitem{Villani:03}
C.~Villani.
\newblock {\em Topics in Optimal Transportation}, volume~58 of {\em Graduate
  Studies in Mathematics}.
\newblock American Mathematical Society, 2003.

\bibitem{Villani:09}
C.~Villani.
\newblock {\em Optimal Transport: Old and New}, volume 338 of {\em Grundlehren
  der mathematischen Wissenschaften}.
\newblock Springer Verlag, 2009.

\bibitem{Wasserstein_Bjorn}
Y.~Yang, B.~Engquist, J.~Sun, and B.~F. Hamfeldt.
\newblock Application of optimal transport and the quadratic {W}asserstein
  metric to full-waveform inversion.
\newblock {\em Geophysics}, 83(1):R43--R62, 2018.

\end{thebibliography}

\end{document}